\documentclass[%
]{mpi2015-cscpreprint}
\usepackage[american]{babel}
\usepackage{amsthm}
\usepackage{amsfonts}
\usepackage{nccmath}
\usepackage[utf8]{inputenc}
\makeatletter
\newcommand*\rel@kern[1]{\kern#1\dimexpr\macc@kerna}
\newcommand*\widebar[1]{%
  \begingroup
  \def\mathaccent##1##2{%
    \rel@kern{0.8}%
    \overline{\rel@kern{-0.8}\macc@nucleus\rel@kern{0.2}}%
    \rel@kern{-0.2}%
  }%
  \macc@depth\@ne%
  \let\math@bgroup\@empty\let\math@egroup\macc@set@skewchar%
  \mathsurround\z@ \frozen@everymath{\mathgroup\macc@group\relax}%
  \macc@set@skewchar\relax
  \let\mathaccentV\macc@nested@a%
  \macc@nested@a\relax111{#1}%
  \endgroup
}
\makeatother

\DeclareMathOperator*{\argmax}{argmax}
\DeclareMathOperator*{\range}{Range\,}
\DeclareMathOperator*{\argmin}{argmin}
\DeclareMathOperator*{\trace}{trace}

\newcommand{\tran}{\ensuremath{\mkern-1.5mu\mathsf{T}}}

\usepackage{MnSymbol}
\usepackage[linesnumbered,ruled,vlined]{algorithm2e}
\usepackage{caption}

\usepackage{amsmath}
\usepackage{cite}
\usepackage{mathtools}
\usepackage{tikz}
\usepackage{pgfplots}
\usepackage{pgfplotstable}
\usepackage{geometry}
\usepackage{booktabs}
\usepackage{framed}

\pgfplotstableset{
	every head row/.style={before row=\toprule,after row=\midrule},
	clear infinite
}

\usepackage{amsopn}

\renewcommand{\leq}{\leqslant}
\renewcommand{\geq}{\geqslant}

\usepackage{pifont}
\newcommand{\cmark}{\textcolor{green}{\ding{51}}}%
\newcommand{\xmark}{\textcolor{red}{\ding{55}}}%

\usepackage{todonotes}

\usepackage[symbol]{footmisc}

\usepackage{multirow}
\usepackage{scalerel,stackengine}
\stackMath{}
\newcommand\reallywidehat[1]{%
  \savestack{\tmpbox}{\stretchto{%
      \scaleto{%
        \scalerel*[\widthof{\ensuremath{#1}}]{\kern-.6pt\bigwedge\kern-.6pt}%
        {\rule[-\textheight/2]{1ex}{\textheight}}
      }{\textheight}%
    }{0.5ex}}%
  \stackon[1pt]{#1}{\tmpbox}%
}
\newcounter{mymac@matlab}
\newcommand{\MATLAB}{MATLAB%
  \ifnum\value{mymac@matlab}<1%
  \textsuperscript{\textregistered}%
  \setcounter{mymac@matlab}{1}%
  \fi%
}

\DeclareOldFontCommand{\rm}{\normalfont\rmfamily}{\mathrm}
\DeclareOldFontCommand{\sf}{\normalfont\sffamily}{\mathsf}
\DeclareOldFontCommand{\tt}{\normalfont\ttfamily}{\mathtt}
\DeclareOldFontCommand{\bf}{\normalfont\bfseries}{\mathbf}
\DeclareOldFontCommand{\it}{\normalfont\itshape{}}{\mathit} 
\DeclareOldFontCommand{\sl}{\normalfont\slshape{}}{\@nomath\sl}
\DeclareOldFontCommand{\sc}{\normalfont\scshape}{\@nomath\sc}

\newtheorem{Theorem}{Theorem}[section]
\newtheorem{Cor}{Corollary}[section]

\newtheorem{remark}[section]{Remark}

\newcounter{example}[section]
\newenvironment{example}[1][]{\refstepcounter{example}\par\medskip
  \textbf{Experiment~\theexample{}} }{}

\usepackage{tikz}
\usepackage{pgfplots} 
\pgfplotsset{
  compat = 1.16, 
  every axis/.append style={%
    cycle list name=mylist,
    xminorticks=true,
    yminorticks=true,
    major grid style={line width=.2pt,draw=black!10},
  },%
  every axis plot/.append style={%
    line width=1.0pt,
    mark size=3pt
  },%
  unbounded coords = jump, 
} 
\usepgfplotslibrary{external} 
\tikzset{external/system call = {%
    lualatex \tikzexternalcheckshellescape%
    -halt-on-error
    -interaction=batchmode
    -jobname "\image" "\texsource"}}
\tikzexternaldisable{} 

\usepackage{filemod}

\definecolor{mycolor1}{HTML}{66C2A5}
\definecolor{mycolor2}{HTML}{FC8D62}
\definecolor{mycolor3}{HTML}{8DA0CB}

\usetikzlibrary{plotmarks}
\pgfplotscreateplotcyclelist{mylist}{
{mycolor1, mark options={solid}, mark=*, solid},
{mycolor2, mark options={solid}, mark=triangle*, dashed},
{mycolor3, mark options={solid}, mark=x, dashdotdotted},
{mycolor1, only marks, mark=-},
{mycolor2, only marks, mark=-},
{mycolor3, only marks, mark=-}
}

\begin{document}
\title{\textcolor{black}{On an integrated Krylov-ADI solver for large-scale Lyapunov equations}}

\author[$\dagger$]{Peter Benner}
\affil[$\dagger$]{Research Group Computational Methods in Systems
  and Control Theory (CSC),
  Max Planck Institute for Dynamics of Complex Technical Systems, Sandtorstra\ss{e} 1, 39106 Magdeburg, Germany.\authorcr%
  \email{benner@mpi-magdeburg.mpg.de}, \orcid{0000-0003-3362-4103}}

\author[$\ast$]{Davide Palitta}
\affil[$\ast$]{Universit\`a di Bologna, Centro $AM^2$, Dipartimento di Matematica, Piazza di Porta S.
Donato 5, 40127 Bologna, Italy.
 \authorcr%
  \email{davide.palitta@unibo.it
  }, \orcid{0000-0002-6987-4430}}

\author[$\S$]{Jens Saak}
\affil[$\S$]{Research Group Computational Methods in Systems
  and Control Theory (CSC),
  Max Planck Institute for Dynamics of Complex Technical Systems, Sandtorstra\ss{e} 1, 39106 Magdeburg, Germany.\authorcr%
  \email{saak@mpi-magdeburg.mpg.de}, \orcid{0000-0001-5567-9637}}

\shorttitle{\textcolor{black}{EKSM-ADI for large-scale Lyapunov equations}}
\shortauthor{P. Benner, D. Palitta, J. Saak }


\abstract{
  One of the most computationally expensive steps of the low-rank ADI method for large-scale Lyapunov equations is the solution of a shifted linear system at each iteration. We propose the use of the extended Krylov subspace method for this task. In particular, we illustrate how a single approximation space can be constructed to solve all the shifted linear systems needed to achieve a prescribed accuracy in terms of Lyapunov residual norm.
  Moreover, we show how to fully merge the two iterative procedures in order to
  obtain a novel, efficient implementation of the low-rank ADI method, for an important class of equations.
  Many state-of-the-art algorithms for the shift computation can be easily incorporated into our new scheme, as well. Several numerical results illustrate the potential of our novel procedure when compared to an implementation of the low-rank ADI method based on sparse direct solvers for the shifted linear systems.
}



\novelty{Extended Krylov subspace methods (EKSM) and the low-rank alternating
  directions implicit (LR-ADI) iteration have been competing methods for the
  solution of large-scale algebraic Lyapunov equations. In this paper, we make
  an important step towards a new method merging them into a combined procedure
  that inherits advantages from both worlds.}

\maketitle
\setcounter{footnote}{7}


\section{Introduction}\label{Introduction}

The low-rank alternating direction implicit
(LR-ADI)~\cite{Li2002,Penzl2000a} method is one of the state-of-the-art
methods for the numerical solution of large-scale Lyapunov equations~\cite{Benner2013,Simoncini2014}.
This linear matrix equation can be encountered in many applications: control and system theory~\cite{Son98,KalFA69}, especially some model reduction techniques for dynamical systems~\cite{Benner2005,Antoulas.05}, 
but also discretization of certain partial differential equations (PDEs)~\cite{Sta91}, and many more. 

We consider Lyapunov equations of the form
\begin{equation}\label{eq:lyap_gen}
  AXE^{\tran}+EXA^{\tran}+BB^{\tran}=0,
\end{equation}
where $A, E\in\mathbb{R}^{n\times n}$, and $B\in\mathbb{R}^{n\times q}$, $q\ll
n$. Moreover, $E$ is supposed to be symmetric positive definite (SPD) and the
matrix pencil 
$(A,E)$ to be asymptotically stable, i.e., its spectrum is contained in the open
left half plane $\mathbb{C}_-$,  which guarantees that a unique solution $X$ exists, it is symmetric positive semidefinite~\cite{Pen97}.

A special case of equation~\eqref{eq:lyap_gen} is attained whenever $E=I$, namely the equation of interest is
\begin{equation}\label{eq:lyap}
  AX+XA^{\tran}+BB^{\tran}=0.
\end{equation}

Oftentimes the coefficient matrix $E$ 
possesses a structured sparsity pattern. For instance, it is (block) diagonal when the matrices stem from a finite element discretization that uses mass-lumping. In this case, we can easily transform equation~\eqref{eq:lyap_gen} and obtain an equation of the form~\eqref{eq:lyap}. This can, for example, be achieved by simply pre- and post-multiplying~\eqref{eq:lyap_gen} by $E^{-\frac{1}{2}}$ to potentially preserve symmetry of \(A\). For the sake of simplicity, we thus focus on equation~\eqref{eq:lyap} in the following. 

In case of very large problem dimensions, the solution $X$ cannot be stored since this matrix is, in general, dense. However, it is well-known that its singular values quickly decay to zero under suitable assumptions, see, e.g.,~\cite{Penzl2000,Baker2015,Gra04,morBenKS16},
 so that accurate low-rank approximations $ZZ^{\tran}\approx X$, $Z\in\mathbb{R}^{n\times t}$, $t\ll n$, can be constructed.
The efficient computation of the low-rank factor $Z$ is the task of
LR-ADI and of all other \emph{low-rank methods}. See, e.g.,  the survey papers~\cite{Simoncini2014,Benner2013} for further details on different low-rank methods for linear matrix equations.

It is well-known that the convergence rate of the LR-ADI method is strictly
connected to the selection of some parameters
${\{p_i\}}_{i=1,\ldots,j}\subset\mathbb{C}_-$ called \emph{shifts}\footnote{We only consider
  \emph{proper} sets of shifts, namely ${\{p_i\}}_{i=1,\ldots,j}$ is closed with
  respect to complex conjugation.}. The computation of effective shifts is a
highly non-trivial task and it has been a rather active research topic in the
last decades. Many strategies are available in the literature and
these can be divided into two categories: \emph{Offline}
routines~\cite{Penzl2000a,Sabino2006,Wachspress2013}, where the shifts are
computed a-priori, before LR-ADI starts and then, potentially, cyclically reused,
and \emph{online} schemes~\cite{Benner2014-2015,Kuerschner2019}, where the
shifts are computed on the fly within the iterative procedure. 
The name shifts for the values $p_j$ comes from the fact that in each LR-ADI
iteration we need to solve a shifted
linear system with a coefficient matrix of the form $A+p_{j}E$, or $A+p_{j}I$, in case of~\eqref{eq:lyap_gen}, or~\eqref{eq:lyap}, respectively. {\color{black} Notice that since $(A,E)$ (or $A$ in case of~\eqref{eq:lyap}) is asymptotically stable and $\{p_j\}\subset\mathbb{C}_-$, all the linear systems involved in the LR-ADI scheme are well defined.} 

In Algorithm~\ref{LR_ADI_algorithm} we report an implementation of the LR-ADI scheme for the solution of~\eqref{eq:lyap_gen}.
 Notice that Algorithm~\ref{LR_ADI_algorithm} is designed to drastically reduce the amount of complex arithmetic
calculations. Indeed, even
though $A$ and $B$ in~\eqref{eq:lyap} are real, the shifts $p_j$ are often
complex if $A$ is nonsymmetric, so that complex arithmetic may
occur. See~\cite{Benner2013b},~\cite[Chapter 4]{Kuerschner2016} and references
therein for details and derivations. 
\begin{algorithm}[t]
  \DontPrintSemicolon{}
  \SetKwInOut{Input}{input}\SetKwInOut{Output}{output}
  \Input{$A\in\mathbb{R}^{n\times n}$ stable, $B\in\mathbb{R}^{n\times q}$,
    max.\ iteration count $j_{\max}$,\newline normalized residual bound $\varepsilon>0$.}
  \Output{$Z_{j}\in\mathbb{R}^{n\times jq}$, $Z_{j}Z_{j}^{\tran}=X_{j}\approx X$ approximate solution to~\eqref{eq:lyap}.}
  \BlankLine{}
  Set $W_0=B$, $Z_0=[]$, $j=1$, and select $p_1\in\mathbb{C}_-$\;
  \While{$\|W_{j-1}^*W_{j-1}\|_F\geq\varepsilon\|B^*B\|_F$ {\rm \textbf{and}} $j\leq j_{\max}$}{
    Solve  $(A+p_{j}I)S_{j}=W_{j-1}$\label{line_linearsolve}\;
    Set  $W_j=W_{j-1}-2\text{Re}(p_j)S_j$\;
    \If{$\text{{\rm Im}}(p_j)\neq0$}{
    Set $\beta=\text{Re}(p_j)/\text{Im}(p_j)$ and $p_{j+1}=\widebar p_j$\;
    Set $S_{j+1}=\sqrt{2(\beta^2+1)}\text{Im}(S_j)$\;
    Set $W_{j+1}=W_{j-1}-4\text{Re}(p_j)(\text{Re}(S_j)+\beta\text{Im}(S_j))$\;
    Set $S_{j}=\sqrt{2}(\text{Re}(S_j)+\beta\text{Im}(S_j))$\;
    Set $Z_{j+1}=[Z_{j-1},\sqrt{-2\text{Re}(p_j)}S_j,\sqrt{-2\text{Re}(p_{j+1})}S_{j+1}]$\;
    Set $j=j+1$\;
    }
    \Else{
    Set $Z_j=[Z_{j-1},\sqrt{-2\text{Re}(p_j)}S_j]$\;
    }
    Choose the next shift $p_{j+1}\in\mathbb{C}_-$
    \;
    Set $j=j+1$\;
  }
  \caption{LR-ADI for Lyapunov equations.}%
  \label{LR_ADI_algorithm}
\end{algorithm}

One of the most computationally expensive steps of
Algorithm~\ref{LR_ADI_algorithm} is the solution of the shifted linear systems
with $q$ right-hand sides in line~\ref{line_linearsolve}.
Such a job has to be
carried out at each LR-ADI iteration. In this contribution, we propose to employ state-of-the-art block Krylov subspace methods for this task. In particular, for equation~\eqref{eq:lyap}, we illustrate how to efficiently reuse the approximation space employed at the $j$-th LR-ADI iteration
and utilize it also in the next one. To this end, it is crucial that the right-hand side of the linear system we need to solve at the \mbox{$(j+1)$-st} iteration can be represented in terms of the basis of the subspace employed in the previous iteration. This simple but 
critical observation lets us design a novel, efficient procedure that can lead to noticeable savings in the running time for the solution of~\eqref{eq:lyap}. Indeed, all the LR-ADI steps can be completely merged into the Krylov routine so that the LR-ADI iteration is only implicitly performed.
Moreover, also the LR-ADI shift computation can be incorporated into the framework proposed in this paper.

The following is a synopsis of the paper. Section~\ref{Block Krylov methods for shifted linear systems} is devoted to recalling the general (block) Krylov subspace framework for shifted linear systems. In particular, some details about the extended Krylov subspace method presented in~\cite{Simoncini2010} are given in Section~\ref{The extended Krylov subspace method for shifted linear systems}.
In Section~\ref{Merging the two iterative procedures} we present the main contribution of the paper and we show how to fully merge the LR-ADI iteration into the projection method adopted for the linear system solution. The selection of effective shifts is crucial for attaining a fast convergence in terms of number of LR-ADI iterations, and numerous strategies have been proposed in the literature to accomplish this task; see, e.g.,~\cite{Sta91,Penzl2000a,Sabino2006,BenMS08,Wachspress2013,Benner2014-2015,Kuerschner2019,Saa09}. In Section~\ref{Shift computation} we illustrate how many of these routines can be integrated into our novel framework with no additional cost. The potential of our strategy is depicted in Section~\ref{Numerical examples}, where several numerical results are reported. We close the paper with our conclusions in Section~\ref{Conclusions}. 

Throughout the paper, we adopt the following notation.
The matrix inner product is
defined as $\langle X, Y \rangle_F \colon= \trace(Y^{\tran} X)$ so that the induced norm is $\|X\|_F^2= \langle X, X\rangle_F$. The Kronecker product is denoted by $\otimes$ whereas
$I_n$ and $O_{n\times m}$ denote the identity matrix of order $n$ and the $n\times m$ zero matrix, respectively. Only one subscript is used for a square zero matrix, i.e., $O_{n\times n}=O_n$, and the subscript
is omitted whenever the dimension of $I$ and $O$ is clear from the context. Moreover, $e_i$ is the $i$-th basis vector of the canonical basis of $\mathbb{R}^n$. 
The brackets $[\cdot]$ are used to concatenate matrices of conformal dimensions. In particular, a MATLAB-like
notation is adopted and $[M,N]$ denotes the matrix obtained by putting $M$ on the left of $N$ whereas $[M;N]$ the one obtained by putting $M$ on top of $N$, i.e., $[M;N]={[M^{\tran},N^{\tran}]}^{\tran}$.
If $w\in\mathbb{R}^{n}$, $\text{diag}(w)$ denotes the $n\times n$ diagonal matrix whose $i$-th diagonal entry corresponds to the $i$-th component of $w$. Given $X\in\mathbb{C}^{n\times m}$, we write $X=\text{Re}(X)+\imath\text{Im}(X)$, where $\text{Re}(X)$ and $\text{Im}(X)$ are its real and imaginary parts, respectively, and $\imath$ is the imaginary unit. The complex conjugate of $X$ is denoted by $\widebar X= \text{Re}(X)-\imath\text{Im}(X)$.




\section{Block Krylov methods for shifted linear systems}\label{Block Krylov methods for shifted linear systems} The literature about the numerical solution of shifted linear systems by Krylov subspace methods is rather vast. Indeed, sequences of shifted linear systems arise in many applications belonging to different research areas like control theory~\cite{Datta1991,Laub1985}, wave propagation problems~\cite{Baumann2017}, mechanical systems~\cite{Feriani2000}, quantum chromodynamics~\cite{GLaeSSNER1996}, and many more. 

This algebraic problem is trickier than it looks and many researchers have contributed to its understanding providing important insights on its properties and designing efficient, robust algorithms for its solution. Here is an incomplete list of contributions on numerical schemes for sequences of shifted linear systems and their analysis~\cite{Simoncini2003,Freund1993,Frommer1998,Meerbergen2003,Soodhalter2014,Soodhalter2016,Soodhalter2016a,Baumann2015}.

{\color{black} In this section, we consider sequences of shifted linear systems of the form
\begin{equation}\label{eq:shifted_linearsystems}
  (A+p_{j}I)S_{j}=W,\quad W\in\mathbb{R}^{n\times q},
\end{equation}
where the right-hand side $W$ does not depend on the index $j$, even though, in line~\ref{line_linearsolve} of Algorithm~\ref{LR_ADI_algorithm}, $W_{j-1}$ does change at every LR-ADI iteration. In Section~\ref{Merging the two iterative procedures} we show how to adapt the machinery, presented here, to the case of linear systems of the form $(A+p_{j}I)S_j=W_{j-1}$, arising within the LR-ADI scheme.}

Any Krylov routine {\color{black} for~\eqref{eq:shifted_linearsystems} computes a numerical solution of the form $S_m^{(j)}=S_0+V_{m}Y_{m}^{(j)}\approx S_{j}$, $V_{m}=[\mathcal{V}_1,\ldots,\mathcal{V}_m]\in\mathbb{R}^{n\times m\ell q}$, $\ell\geq 1$\footnote{{\color{black}The value of $\ell$ depends on the adopted approximation space. It holds, $\ell=1$ for the polynomial Krylov subspace in~\eqref{def:polynomialKrylov}, whereas $\ell=2$ for the extended Krylov subspace in~\eqref{def:extendedKrylov}.}}, $\mathcal{V}_i\in\mathbb{R}^{n\times \ell q}$, $i=1,\ldots,m$, $Y_m^{(j)}\in\mathbb{C}^{m\ell q\times q}$}, where the orthonormal columns of $V_m$ span a suitable subspace $\mathcal{K}_m$, namely, $\range(V_m)=\mathcal{K}_m$, $S_0$ is an initial guess, and the matrix $Y_m^{(j)}$ can be computed by imposing different conditions. In particular, $Y_m^{(j)}$ is often computed by either imposing a Galerkin condition on the residual or minimizing the residual norm. For the sake of simplicity, we consider $S_0=O$ in the following.

One of the most common choices for the approximation space $\mathcal{K}_m$ 
is the block Krylov subspace
\begin{equation}\label{def:polynomialKrylov}
  \mathbf{K}_m^\square(A,W)=\range([W,AW,\ldots,A^{m-1}W]).
\end{equation}
See, e.g.,~\cite{OLeary1980,Schmelzer2004,Frommer2017} and the references therein for further details on the block polynomial Krylov subspace $\mathbf{K}_m^\square(A,W)$ and related methods.

However, Simoncini showed in~\cite{Simoncini2010} that the extended Krylov subspace~\cite{DruK98} 
\begin{equation}\label{def:extendedKrylov}
  \mathbf{EK}_m^\square(A,W)=\range([W,A^{-1}W,AW,A^{-2}W,\ldots,A^{m-1}W,A^{-m}W]),
\end{equation}
can be a powerful alternative for the solution of~\eqref{eq:shifted_linearsystems} in many cases. For instance, when $A$ is large and real while the $p_j$'s are complex. See also Section~\ref{The extended Krylov subspace method for shifted linear systems}.

The basis $V_m$ of both the polynomial and extended Krylov subspace can be constructed by means of the (extended) Arnoldi process and the following Arnoldi relation is fulfilled
\begin{equation}\label{eq:Arnoldi}
  AV_{m}=V_{m}T_{m}+\mathcal{V}_{m+1}E_{m+1}^{\tran}\underline{T}_{m},
\end{equation}
where $\underline{T}_m=V_{m+1}^{\tran}AV_m\in\mathbb{R}^{(m+1)\ell q\times m\ell q}$, $T_m$ is its principal square submatrix, and $E_{m+1}=e_{m+1}\otimes I_{\ell q}$. See, e.g.,~\cite{Saad2003,Simoncini2007}.

The Arnoldi relation~\eqref{eq:Arnoldi} is one of the most crucial tools in the
solution of~\eqref{eq:shifted_linearsystems} by Krylov methods. Indeed, it can
be used to show the fundamental \emph{shift-invariance} property of the Krylov
subspaces~\eqref{def:polynomialKrylov} and~\eqref{def:extendedKrylov}, and the
following relation holds true
\begin{equation}\label{eq:shiftedArnoldi}
  (A+p_{j}I_{n})V_{m}=V_{m}(T_{m}+p_{j}I_{m\ell q})+\mathcal{V}_{m+1}E_{m+1}^{\tran}\underline{T}_{m}.
\end{equation}
See, e.g.,~\cite[Equation (2.1)]{Simoncini2003},~\cite[Equation (3.1)]{Simoncini2010}.

Equation~\eqref{eq:shiftedArnoldi} says that we can compute only one approximation space for solving~\eqref{eq:shifted_linearsystems}. In particular, the space constructed using $A$, i.e., $\mathbf{K}_m^\square(A,W)$ or $\mathbf{EK}_m^\square(A,W)$, can be employed, by possibly being expanded, to solve all the shifted linear systems in the sequence~\eqref{eq:shifted_linearsystems}.

Polynomial Krylov subspace methods often need many iterations to achieve the prescribed accuracy, so that a large subspace is constructed. This leads to an increment in both the storage demand and the computational efforts of the selected solution procedure.  
Different strategies have been developed to avoid the construction of a too large subspace. 

With the goal of achieving a fast convergence in terms of number of iterations, the linear system~\eqref{eq:shifted_linearsystems} can be preconditioned, namely
is transformed into an equivalent problem with better spectral properties. However, designing effective preconditioning operators for a sequence of shifted linear systems is a difficult task and often highly problem dependent. Very sophisticated schemes have been proposed in the literature. See, e.g.,~\cite{Luo2014,Bellavia2011,Benzi2003,Bertaccini2004,Anzt2016}. 

Restarted routines are an alternative solution. In this framework, the
approximation space $\mathcal{K}_m$ is expanded until it reaches a prescribed
maximum dimension. If the desired level of accuracy is not achieved, the last
computed basis block $\mathcal{V}_{m+1}$ is employed as initial block in the
construction of a new subspace $\mathcal{K}_m'$. This procedure is iterated
until a stopping criterion is fulfilled. See, e.g.,~\cite{Simoncini2003,
  Frommer1998} and~\cite[Section 3.2.1]{Frommer2017}. However, in our framework the LR-ADI shifts $p_j$'s are often computed on the fly and, thus,
are not all available at the same time. 
Therefore, to fully take advantage of the computational efforts needed to solve the linear system $(A+p_{j-1}I)S_{j-1}=W$, we would have to store all the bases computed during the employed restarted Krylov procedure and use them to solve the $j$-th linear system, as well. Unfortunately, this would destroy all the benefits in terms of storage complexity gained from the restart-paradigm.

In~\cite{Simoncini2010}, Simoncini showed that the employment of the extended Krylov subspace~\eqref{def:extendedKrylov}, in place of~\eqref{def:polynomialKrylov}, often leads to a faster convergence, in terms of iterations, to the point that the constructed subspace is usually smaller than the polynomial counterpart needed to reach the same level of accuracy. 
We, thus, decide to use such an approximation space for the solution of the shifted linear systems within the LR-ADI method and in the next section we recall some details of the extended Krylov subspace method.

Notice that the faster convergence of the extended Krylov subspace~\eqref{def:extendedKrylov} comes with a toll. Indeed, at each iteration, a linear system with $A$ has to be solved during the basis construction. Nevertheless, the increase in the overall workload of the solution process can be limited in general. Indeed, if we want to use a direct solver to invert $A$, for instance, the LU factors of $A$ can be computed once and for all before the LR-ADI scheme starts. On the other hand, if an iterative procedure is employed, analogously a single preconditioner for $A$ has to be designed once.

{\color{black}As already mentioned,} in the formulation~\eqref{eq:shifted_linearsystems} the right-hand side $W$ is fixed, namely it does not depend on the shift index $j$. However, in line~\ref{line_linearsolve} of Algorithm~\ref{LR_ADI_algorithm},  the linear systems we need to solve are of the form 
\[(A+p_{j}I)S_{j}=W_{j-1}.\]
At a first glance, having a nonconstant right-hand side does not allow for the employment of the shifted Krylov framework we briefly described above. A larger class of solvers, the so-called \emph{recycling} Krylov methods, seems more appropriate. See, e.g.,~\cite{Soodhalter2016,Parks2006,Soodhalter2014,Gaul2014,SooDK20} for general sequences of shifted linear systems, and~\cite{Ahuja2012,Ahuja2015,morFenBK13} for some recycling Krylov techniques applied in a model reduction context. However, in Section~\ref{Merging the two iterative procedures} we show that, in the LR-ADI context for $j>1$, the residual factor $W_{j-1}$ belongs to the subspace $\mathcal{K}_m$ employed in the solution of the \((j-1)\)-st linear system $(A+p_{j-1}I)S_{j-1}=W_{j-2}$. Along with the shift-invariance property of the Krylov subspace, this observation allows us to utilize only one subspace for the solution of all the shifted linear systems within the LR-ADI method. In turn, as shown in Section~\ref{Numerical examples}, we can notably reduce the computational effort of the overall procedure.

\subsection{The extended Krylov subspace method for shifted linear systems}\label{The extended Krylov subspace method for shifted linear systems} In this section, we recall the extended Krylov subspace method for shifted linear systems presented in~\cite{Simoncini2010}.

Given the sequence of shifted linear systems~\eqref{eq:shifted_linearsystems}, the extended Krylov subspace method computes a solution of the form $S_m^{(j)}=V_{m}Y_{m}^{(j)}$, where the $2mq$ orthonormal columns of $V_m$ span the extended Krylov subspace~\eqref{def:extendedKrylov}, whereas the $2mq\times q$ matrix $Y_m^{(j)}$ can be computed in different manners.

For instance, $Y_m^{(j)}$ can be computed by imposing a Galerkin condition on the residual $R_m^{(j)}=(A+p_{j}I)V_{m}Y_{m}^{(j)}-W$, namely by imposing $V_{m}^{\tran}R_{m}^{(j)}=0$. Thanks to the shifted Arnoldi relation~\eqref{eq:shiftedArnoldi}, it is easy to show that such a Galerkin condition is equivalent to solving the projected linear systems
\begin{equation}\label{eq:projected_linearsystem}
  (T_{m}+p_{j}I)Y_{m}^{(j)}=E_1\gamma,
\end{equation}
where $E_1=e_1\otimes I_{2q}$, and $\gamma\in\mathbb{R}^{2q\times q}$ is such that $W=V_1\gamma$.

With $Y_m^{(j)}$ at hand, the Frobenius norm of the residual $\|R_m^{(j)}\|_F$ can be computed at low cost, as
\begin{equation}\label{eq:residual_norm}
  \|R_{m}^{(j)}\|_F=\|E_{m+1}^{\tran}\underline{T}_{m}Y_{m}^{(j)}\|_F, 
\end{equation}
following~\cite[Equation~(3.2)]{Simoncini2010}.

Alternatively, following the discussion in~\cite[Section 4.1]{Soodhalter2016a}, the matrix $Y_{m}^{(j)}$ can be computed also by minimizing the residual norm, i.e., 
\begin{equation}\label{eq:minimalresidual}
  Y_{m}^{(j)}=\argmin_{Y\in\mathbb{R}^{2mq\times q}}\|(A+p_{j}I)V_{m}Y-W\|_{F}.
\end{equation}
Once again, thanks to the shifted Arnoldi relation~\eqref{eq:shiftedArnoldi}, the minimization problem in~\eqref{eq:minimalresidual} simplifies, and we can compute 
$Y_m^{(j)}$ as 
\begin{equation}\label{eq:projected_minimalresidual}
  Y_m^{(j)}=\argmin_{Y\in\mathbb{R}^{2mq\times q}}\|(\underline{T}_m+p_j[I_{2mq};O_{2q\times 2mq}] )Y-E_1\gamma\|_F.
\end{equation}
{\color{black} Note the abuse of notation in~\eqref{eq:projected_minimalresidual}: $E_1\in\mathbb{R}^{2(m+1)q\times 2q}$ whereas $E_1\in\mathbb{R}^{2mq\times 2q}$ in~\eqref{eq:projected_linearsystem}.
}

If $QP=\underline{T}_m+p_j[I_{2mq};O_{2q\times 2mq}]$ denotes the QR factorization of 
$\underline{T}_m+p_j[I_{2mq};O_{2q\times 2mq}]$, and we consider the following partition
\[
  Q=[Q_1,Q_2],\,Q_1\in\mathbb{R}^{2(m+1)q\times 2mq},\, Q_2\in\mathbb{R}^{2(m+1)q\times 2q},\;
  P=\begin{bmatrix}
    P_1\\
    O_{2q\times 2mq}\\
  \end{bmatrix},\, P_1\in\mathbb{R}^{2mq\times 2mq},
\] 
then the matrix $Y_{m}^{(j)}$ in~\eqref{eq:projected_minimalresidual} can be computed as
\begin{equation}\label{eq:ComputeY_MR}
  Y_m^{(j)}=P_1^{-1}Q_1^{\tran}E_1\gamma,\end{equation}
and the residual norm is given by 
\begin{equation}\label{eq:residual_norm_MR}
  \|R_m^{(j)}\|_F=\|Q_2^{\tran}E_1\gamma\|_F.
\end{equation}
The overall procedure is summarized in Algorithm~\ref{EKSM_shiftedlinearsystems}, {\color{black} where $\Sigma$ contains the indices of all yet unsolved systems, whereas $\Sigma_C$ contains the indices of all the systems that have already been solved. 
The basis block $\mathcal V_{m+1}$ can be computed by following~\cite{Simoncini2007}. This operation involves both matrix-vector products and linear system solves with $A$. Moreover,} the basis $V_m$ is real whenever $A$ and $W$ are so. Complex arithmetic may occur in the computation of $Y_m^{(j)}$, if $\text{Im}(p_j)\neq 0$.

Notice that as soon as the $j$-th linear system has converged, namely the related relative residual norm is sufficiently small, we stop solving the $j$-th projected problem\footnote{Either~\eqref{eq:projected_linearsystem} or~\eqref{eq:projected_minimalresidual}.}. Once all the linear systems have converged, we terminate the iterative process.

To conclude, we would like to point out that, to the best of our knowledge, this is the first time the minimal residual condition~\eqref{eq:projected_minimalresidual} is proposed within the extended Krylov subspace method for shifted linear systems.

\begin{algorithm}[t]
  \DontPrintSemicolon{}
  \caption{Extended Krylov subspace method for shifted linear systems.}%
  \label{EKSM_shiftedlinearsystems}
  \SetKwInOut{Input}{input}\SetKwInOut{Output}{output}
  \Input{$A\in\mathbb{R}^{n\times n}$, $\{p_j\}\subset\mathbb{C}$, $W\in\mathbb{R}^{n\times q}$, and normalized residual bound $\varepsilon>0$}
  \Output{$S_m^{(j)}\in\mathbb{R}^{n\times q}$,
    $S_m^{(j)}=V_{m}Y_{m}^{(j)}\approx S_{j}$, {\color{black} where $(A+p_{j}I)S_{j}=W_{j}$}}
  \BlankLine%
  Set $\beta=\|W\|_F$, $\Sigma=\{1,\ldots,\max j\}$, $\Sigma_C=\emptyset$\;
  Perform economy-size QR, $[W,A^{-1}W]=[\mathcal{V}_1^{(1)},\mathcal{V}_1^{(2)}][ \gamma,\theta]$, $\gamma,\theta\in\mathbb{R}^{2q\times q}$\; 
  \For{$m=1, 2,\dots$, till convergence,}{
    Compute next basis block $\mathcal{V}_{m+1}$ as in~\cite{Simoncini2007} and set $V_{m+1}=[V_{m},\mathcal{V}_{m+1}]$ \;\label{Alg.line:basis}
    Update $T_m=V_m^{\tran}AV_m$ as in~\cite{Simoncini2007}\;
    Compute $Y_m^{(j)}$ for $j\in\Sigma\diagdown\Sigma_C$ as in~\eqref{eq:projected_linearsystem} or~\eqref{eq:ComputeY_MR}\;
    Compute $\|R_m^{(j)}\|_F$ accordingly as in~\eqref{eq:residual_norm} or~\eqref{eq:residual_norm_MR}  \;

    \If{$\|R_m^{(j)}\|_F/\beta<\varepsilon$}{
      Set $\Sigma_C=\Sigma_C\cup\{j\}$\; } 
    \If{$\Sigma \diagdown\Sigma_C=\emptyset$}{
      \textbf{Break} and go to \textbf{\ref{Alg.line:last}} \;}
  }
  Set $S_m^{(j)}=V_m Y_m^{(j)}$\label{Alg.line:last}\;
\end{algorithm}

\section{Merging the two iterative procedures}\label{Merging the two iterative procedures}
In this section we show how the LR-ADI iteration and the extended Krylov subspace method for shifted linear systems can be merged together into a novel, efficient iterative procedure for the solution of~\eqref{eq:lyap}.

As already mentioned, in the sequence of shifted linear systems in
line~\ref{line_linearsolve} of Algorithm~\ref{LR_ADI_algorithm}, also the
right-hand side $W_{j-1}$ depends on the current LR-ADI iteration
$j$. Therefore, at a first glance, we seemingly have to build a new subspace at each iteration $j$, by employing the current $W_{j-1}$ as initial block. However, in the following theorem we show that $W_{j-1}$ belongs to the subspace constructed to solve the $(j-1)$-st linear system so that such a space can be used, by being possibly expanded, also in the solution of the subsequent linear system.

\begin{Theorem}\label{Th:range_W}
 Let $S_{j}=V_{m_{j}}Y_{m_{j}}$, $j\geq 1$, $\range(V_{m_{j}})=\mathbf{EK}_{m_{j}}^\square(A,B)$ for certain $m_{j}\geq 0$. Then 
 \[
   \range(W_{j})\subseteq\mathbf{EK}_{m_{j}}^\square(A,B).
 \]
\end{Theorem}

\begin{proof}
 We are going to show the statement by induction on $j$.

The first linear system to be solved within the LR-ADI method is $(A+p_1I)S_1=B$ and the extended Krylov subspace $\mathbf{EK}_{m_1}^\square(A,B)$ can be employed to this end. The computed solution is of the form $S_1=V_{m_1}Y_{m_1}$, $m_1>0$, where $\range(V_{m_1})=\mathbf{EK}_{m_1}^\square(A,B)$ and $Y_{m_1}\in\mathbb{C}^{2m_1q\times q}$. It is thus easy to show that $W_1=B-2\text{Re}(p_1)S_1=V_{m_1}(E_1\gamma -2\text{Re}(p_1)Y_{m_1})$ is such that $\range(W_1)\subseteq\mathbf{EK}_{m_1}^\square(A,B)$.

We now assume the statement holds for a certain $j-1\geq1$, and we show it holds for $j$ as well.
Since $S_j=V_{m_j}Y_{m_j}$ by assumption and $\range(W_{j-1})\subseteq\mathbf{EK}_{m_{j-1}}^\square(A,B)$ by inductive hypothesis, namely we can write $W_{j-1}=V_{m_{j-1}}\Upsilon_{j-1}$ for a certain $\Upsilon_{j-1}\in\mathbb{R}^{2m_{j-1}q\times q}$, we have 
\begin{eqnarray*}
 W_j=&W_{j-1}-2\text{Re}(p_j)S_j=V_{m_{j-1}}\Upsilon_{j-1}-2\text{Re}(p_j)V_{m_j}Y_{m_j}\\
 =&V_{m_j}\left([\Upsilon_{j-1};O_{2(m_j-m_{j-1})q\times q}]-2\text{Re}(p_j)Y_{m_j}\right).
\end{eqnarray*}
Therefore, $\range(W_{j})\subseteq\mathbf{EK}_{m_{j}}^\square(A,B)$.
\end{proof}
Theorem~\ref{Th:range_W} shows that $W_j$ is exactly represented in $\mathbf{EK}_{m_{j}}^\square(A,B)$. This means that the latter subspace can be still employed for the computation of $S_{j+1}$ by being possibly expanded. Indeed, no components of $W_{j}$ are annihilated when either the Galerkin or the minimal residual condition is imposed. In the following corollary we show how to easily write down the projected problems~\eqref{eq:projected_linearsystem} and~\eqref{eq:projected_minimalresidual} along with the corresponding residual norm computation.
\begin{Cor}
 Assume the prerequisites of Theorem~\ref{Th:range_W} hold. If a Galerkin condition is imposed for the computation of $S_j=V_{m_j}Y_{m_j}$, then the matrix $Y_{m_j}$ amounts to the solution of the projected linear system
\begin{equation}\label{eq:projected_linearsystem_general}
  (T_{m_j}+p_{j}I_{2m_{j}q})Y_{m_{j}}= [\Upsilon_{j-1};O_{2(m_j-m_{j-1})q\times q}],
\end{equation}
where $\Upsilon_{j-1}\in\mathbb{R}^{2m_{j-1}q\times q}$
 is such that $W_{j-1}=V_{m_{j-1}}\Upsilon_{j-1}$, $m_{j-1}\leq m_j$.
The related residual norm can be computed by 
\begin{equation}\label{eq:res_projectedgeneral}
  \|R_{m_j}\|_F=\|E_{m_j+1}^{\tran}\underline{T}_{m_j}Y_{m_j}\|_F.
\end{equation}
Similarly, if a minimal residual norm condition is imposed, we have
\begin{equation}\label{eq:projected_minimalresidual_general}
  Y_{m_j}=\argmin_{Y\in\mathbb{C}^{2m_{j}q\times q}}\|(\underline{T}_{m_{j}}+p_j[I_{2m_{j}q};O_{2q\times 2m_{j}q}] )Y-[\Upsilon_{j-1};O_{2(m_{j}-m_{j-1}+1)q\times q}]\|_F,
\end{equation}
so that
\begin{equation}\label{eq:res_projectedgeneral2}
  \|R_ {m_j}\|_F=\|Q_2^{\tran}[\Upsilon_{j-1};O_{2(m_j-m_{j-1}+1)q\times q}]\|_F,
\end{equation}
where the $2q$ orthonormal columns of $Q_2$ are a basis of the kernel of $\underline{T}_{m_{j}}+p_{j}[I_{2m_{j}q};O_{2q\times 2m_{j}q}]$.
\end{Cor}
\begin{proof}
 Since $W_{j-1}=V_{m_{j-1}}\Upsilon_{j-1}$ and we look for a solution $S_j=V_{m_{j}}Y_{m_j}$ to $(A+p_{j}I)S_{j}=W_{j-1}$, we can write
 \[
 \begin{array}{rcl}
   R_{m_j}&=&(A+p_{j}I)S_j-W_{j-1}=
 (A+p_{j}I)V_{m_{j}}Y_{m_j}-V_{m_{j-1}}\Upsilon_{j-1}\\
 &=&
 V_{m_j}\left((T_{m_j}+p_{j}I_{2m_{j}q})Y_{m_{j}}- [\Upsilon_{j-1};O_{2(m_j-m_{j-1})q\times q}]\right)+\mathcal{V}_{m_j+1}E_{m_j+1}^{\tran}\underline{T}_{m_j}\\
 &=&V_{m_j+1}\left((\underline{T}_{m_j}+p_{j}
 [I_{2m_{j}q};O_{2q\times 2m_{j}q}] 
 )Y_{m_{j}}- [\Upsilon_{j-1};O_{2(m_j-m_{j-1}+1)q\times q}]\right).
 \end{array}
\]
 If a Galerkin condition is imposed, namely $V_{m_j}^{\tran}R_{m_j}=0$,
 then $Y_{m_j}$ is the solution of the linear system in~\eqref{eq:projected_linearsystem_general} and the related residual norm $\|R_{m_j}\|_F$ can be computed as in~\eqref{eq:res_projectedgeneral}.
 
 Similarly, if a minimal residual condition is imposed, $Y_{m_j}$ solves the minimization problem~\eqref{eq:projected_minimalresidual_general} and the $\|R_{m_j}\|_F$ fulfills~\eqref{eq:res_projectedgeneral2}.
\end{proof}
%
%
%
%
%
%

Once $S_j=V_{m_j}Y_{m_j}$ is computed, namely the related residual norm $\|R_{m_j}\|_F$ is sufficiently small, we proceed with the remaining LR-ADI operations.

We would like to point out that the expression of $W_j$, i.e., $W_j=V_{m_j}\Upsilon_{j}$, can be exploited for the Lyapunov residual norm as well. Indeed, 
\begin{equation}\label{eq:lyap_residual_new}
  \|W_j^*W_j\|_F=\|\Upsilon_{j}^*\Upsilon_{j}\|_F.
\end{equation}
This means that also the computation of the Lyapunov residual norm can be carried out by manipulating small matrices of dimension $2m_{j}q\times q$. Similarly, the solution $Z_j$ can be assembled at the very end of the LR-ADI procedure once the residual norm in~\eqref{eq:lyap_residual_new} is sufficiently small. Indeed,
\begin{equation}
 \label{eq:lowrankfactor_sol}
 \begin{array}{rcl}
Z_j&=&[Z_{j-1},\sqrt{-2\text{Re}(p_j)}S_j]=[\sqrt{-2\text{Re}(p_1)}S_1,\sqrt{-2\text{Re}(p_2)}S_2,\ldots,\sqrt{-2\text{Re}(p_j)}S_j]\\
&&\\
  &=&[\sqrt{-2\text{Re}(p_1)}V_{m_1}Y_{m_1},\sqrt{-2\text{Re}(p_2)}V_{m_2}Y_{m_2},\ldots,\sqrt{-2\text{Re}(p_j)}V_{m_j}Y_{m_j}]\\
  &&\\
  &=&V_{m_j}[[Y_{m_1};O_{2(m_j-m_1)q\times q}],[Y_{m_2};O_{2(m_j-m_2)q\times q}],\ldots,Y_{m_j}]\\
  &&\cdot(\sqrt{-2\text{diag}(\text{Re}(p_1),\ldots,\text{Re}(p_j))}\otimes I_q). \end{array}
\end{equation}
The overall procedure combining the LR-ADI iteration with the extended Krylov subspace method for shifted linear systems is depicted in Algorithm~\ref{LR_ADI_EKSM}\footnote{Many subscripts have been removed to make the algorithm more readable.}.

\begin{algorithm}[t]
  \DontPrintSemicolon{}
  \caption{LR-ADI-EKSM for Lyapunov equations.}%
  \label{LR_ADI_EKSM}
  \SetKwInOut{Input}{input}\SetKwInOut{Output}{output}
  \Input{$A\in\mathbb{R}^{n\times n}$, $A$ stable, $B\in\mathbb{R}^{n\times q}$, max.\ inner and outer iteration count $m_{\max}$, $j_{\max}$, normalized residual bound $\varepsilon_{\mathtt{out}}>0$}
  \Output{$Z_{j}\in\mathbb{R}^{n\times jq}$, $Z_{j}Z_{j}^{\tran}\approx X$ approximate solution to~\eqref{eq:lyap}.}
  \BlankLine{}
  Set $\nu=\|B^*B\|_F$, $m_0=0$, $Y_0=O$, and select $p_1\in\mathbb{C}_-$\;
  Perform economy-size QR, $[B,A^{-1}B]=[\mathcal{V}_1^{(1)},\mathcal{V}_1^{(2)}][ \gamma,\theta]$, $\gamma,\theta\in\mathbb{R}^{2q\times q}$\label{basis_construction_1}\; 
  Set $\Upsilon_0=E_1\gamma$, $m=j=1$, and select $\varepsilon_{\mathtt{inn}}^{(1)}$\;
  \While{$m\leq m_{\max}$ {\rm\textbf{and}} $j\leq j_{\max}$}{
    Compute next basis block $\mathcal{V}_{m+1}$ as in~\cite{Simoncini2007} and set $V_{m+1}=[V_{m},\mathcal{V}_{m+1}]$ \;\label{Alg.line:basis2}
    Update $T_m=V_m^{\tran}AV_m$ as in~\cite{Simoncini2007}\label{basis_construction_3}\;
    Compute $Y_{m}$ as in~\eqref{eq:projected_linearsystem_general} or~\eqref{eq:projected_minimalresidual_general}\label{projected_problem_1}
    \;
    Compute $\|R_m\|_F$ accordingly as in~\eqref{eq:res_projectedgeneral} or~\eqref{eq:res_projectedgeneral2}\label{projected_problem_2}\;

    \If{$\|R_m\|_F\leq \varepsilon_{\mathtt{inn}}^{(j)}$}{
      $\mathtt{flag\_noexpand}=1$\;
      \While{$\mathtt{flag\_noexpand}$}{
        Set $m_j=m$ and $Y_{m_j}=Y_m$ \;
        Set $\Upsilon_j=\Upsilon_{j-1}-2\text{Re}(p_j)Y_{m_j}$\;
        \If{$\text{{\rm Im}}(p_j)\neq 0$}{
          Set $m_{j+1}=m_j$, $\beta=\text{Re}(p_j)/\text{Im}(p_j)$, and $p_{j+1}=\widebar p_j$\label{line:complexpj1}\;
          Set $Y_{m_{j+1}}=\sqrt{2(\beta^2+1)}\text{Im}(Y_{m_j})$ \;
          Set $\Upsilon_{j+1}=\Upsilon_{j-1}-4\text{Re}(p_j)(\text{Re}(Y_{m_j})+\beta \text{Im}(Y_{m_j}))$\;
          Set $Y_{m_j}=\sqrt{2}(\text{Re}(Y_{m_j})+\beta \text{Im}(Y_{m_j}))$\;
          Set $j=j+1$\label{line:complexpjend}\;
        }
        \If{$\|\Upsilon_j^*\Upsilon_j\|_F\leq\nu\cdot\varepsilon_{\mathtt{out}}$}{
          \textbf{Break} and go to \textbf{\ref{Alg.line:last2}} \;}
        Choose the next shift $p_{j+1}\in\mathbb{C}_-$\label{shift_computation_1}\;
        Set $j=j+1$\label{line_flag1}\;
        Compute $Y_{m}$ as in~\eqref{eq:projected_linearsystem_general} or~\eqref{eq:projected_minimalresidual_general}\label{projected_problem_1bis}
        \;
        Compute $\|R_m\|_F$ accordingly as in~\eqref{eq:res_projectedgeneral} or~\eqref{eq:res_projectedgeneral2}\label{projected_problem_2bis}\;
        Select $\varepsilon_{\mathtt{inn}}^{(j)}$\;
        \If{$\|R_m\|_F\geq \varepsilon_{\mathtt{inn}}^{(j)}$}{
          $\mathtt{flag\_noexpand}=0$\label{line_flag2}\;
        }
      }
    }
    Set $m=m+1$\;
  } 
  
  {\small$Z_j=V_m [[Y_{m_1};O_{2(m_j-m_1)q\times q}],[Y_{m_2};O_{2(m_j-m_2)q\times q}],\ldots,Y_{m_j}](\sqrt{-2\text{diag}(\text{Re}(p_1),\ldots,\text{Re}(p_j))}\otimes I_q)$}\label{Alg.line:last2}
\end{algorithm}

As in Algorithm~\ref{LR_ADI_algorithm}, if $\text{Im}(p_j)\neq 0$, in  lines~\ref{line:complexpj1} to~\ref{line:complexpjend} we set $p_{j+1}=\widebar p_j$, and we follow the implementation suggested in~\cite{Benner2013b, Kuerschner2016} to reduce the amount of complex arithmetic. In particular, $Y_{m_{j+1}}$ can be obtained from $Y_{m_{j}}$ without solving~\eqref{eq:projected_linearsystem_general} or~\eqref{eq:projected_minimalresidual_general}. Moreover, the adopted scheme results in a real $Z_j$. See~\cite{Benner2013b} and~\cite[Algorithm 4.3]{Kuerschner2016} for further details.
{\color{black}
\begin{remark}
  Theorem~\ref{Th:range_W} shows that $\text{Range}(W_j)\subseteq\mathbf{EK}_{m_j}^\square(A,B)$ whenever $W_j$ is updated as $W_j=W_{j-1}-2\text{Re}(p_j)S_j$, namely whenever all the employed shifts are real. In case of shifts with nonzero imaginary part, the LR-ADI implementation we adopt sets
  \[
    W_{j+1}=W_{j-1}-4\text{Re}(p_j)(\text{Re}(S_j)+\beta\text{Im}(S_j)).
  \]
  Therefore, we need to show that $W_{j+1}$ defined as above is still such that $\text{Range}(W_{j+1})\subseteq\mathbf{EK}_{m_j}^\square(A,B)$. This can be done by applying the same exact arguments used in the proof of Theorem~\ref{Th:range_W}. In particular, the result follows by noticing that the basis $V_m$ is real, as we assumed $A$ and $B$ to be real matrices, and that we can write 
  \begin{align*}
  W_{j+1}=&W_{j-1}-4\text{Re}(p_j)(\text{Re}(S_j)+\beta\text{Im}(S_j))\\
  =&V_{m_j}\left([\Upsilon_{j-1};O_{2(m_j-m_{j-1})q\times q}]-4\text{Re}(p_j)(\text{Re}(Y_{m_j})+\beta\text{Im}(Y_{m_j})\right).
  \end{align*}
 \end{remark}
 }

Notice that two tolerances $\varepsilon_{\mathtt{inn}}^{(j)}$, and $\varepsilon_{\mathtt{out}}$ are employed in Algorithm~\ref{LR_ADI_EKSM}. In particular, $\varepsilon_{\mathtt{out}}$ is used to assess the accuracy of the computed solution in terms of the Lyapunov residual norm, whereas $\varepsilon_{\mathtt{inn}}^{(j)}$ is employed to determine whether the solution of the current linear system is sufficiently correct. In principle, the user can provide a fixed value for the inner tolerance, i.e.\ $\varepsilon_{\mathtt{inn}}^{(j)}\equiv\widebar \varepsilon_{\mathtt{inn}}$ for all $j$.
However, the theory developed in~\cite{Kuerschner2018} can be used to adaptively compute $\varepsilon_{\mathtt{inn}}^{(j)}$ as the LR-ADI iterations proceed. The \emph{relaxation} strategy presented in~\cite[Section 3]{Kuerschner2018} allows us to increase $\varepsilon_{\mathtt{inn}}^{(j)}$ as $j$ grows. Therefore, especially when $\varepsilon_{\mathtt{inn}}^{(j)}$ is rather large, there is no need to expand the current extended Krylov subspace in general. In all the results reported in Section~\ref{Numerical examples}, we employ such a strategy and $\varepsilon_{\mathtt{inn}}^{(j)}$ is computed according to~\cite[Equation (3.18b)]{Kuerschner2018}. See also~\cite{Liu2020} for similar results in case of Sylvester equations.

We would like to point out that the lines~\ref{line_flag1} to~\ref{line_flag2} in Algorithm~\ref{LR_ADI_EKSM} and the use of the flag {\tt flag\_noexpand} are crucial to reduce the computational cost of the overall procedure. Indeed, those lines are devoted to check whether the current subspace already contains enough spectral information to solve the current linear system. If this is the case, we do not expand the current space avoiding unnecessary increments in the memory requirements and computational efforts.

If $\mathbf{Y}:= [[Y_{m_1};O_{2(m_j-m_1)q\times q}],[Y_{m_2};O_{2(m_2-m_1)q\times q}],\ldots,Y_{m_j}]$,~\eqref{eq:lowrankfactor_sol} shows that the numerical solution computed by the proposed LR-ADI implementation is of the form 
\begin{equation}\label{eq:lowrank_sol_ADI}
  Z_{j}Z_{j}^{\tran}=-2V_{m_{j}}(\mathbf{Y}(\text{diag}(\text{Re}(p_{1}),\ldots,\text{Re}(p_{j}))\otimes I_{q}) \mathbf{Y}^{\tran})V_{m_{j}}^{\tran}.
\end{equation}
The right-hand side in~\eqref{eq:lowrank_sol_ADI} has the typical form of an approximate solution computed by a projection method applied to~\eqref{eq:lyap}. In particular, if the extended Krylov subspace method (K-PIK) presented in~\cite{Simoncini2007} is applied to {\color{black} solve}~\eqref{eq:lyap}, the computed approximation is of the form $X_{m}=V_{m}{\color{black}L_{m}}V_{m}^{\tran}$, where the orthonormal columns of $V_{m}$ are a basis of $\mathbf{EK}_{m}^\square(A,B)$ and $L_{m}$ is computed by imposing a Galerkin condition on the residual matrix $AV_{m}{\color{black}L_{m}}V_{m}^{\tran}+V_{m}{\color{black}L_{m}}V_{m}^{\tran}A^{\tran}+BB^{\tran}$. 
Therefore, the proposed LR-ADI implementation can be seen as a novel projection method where the coefficients of the linear combination in terms of the basis vectors that provides the approximate solution, namely the matrix $\mathbf{Y}(\text{diag}(\text{Re}(p_1),\ldots,\text{Re}(p_j))\otimes I_q) \mathbf{Y}^{\tran}$, is computed as outlined above and not by imposing a Galerkin condition on the residual matrix. This perspective may provide new insights on the relation between LR-ADI and K-PIK\@. However, this is beyond the scope of this paper. Similar investigations, relating LR-ADI and rational Krylov subspace methods have been reported in~\cite{Druskin2011,morWol15,morWolP16}.

The expression~\eqref{eq:lowrank_sol_ADI} resembles the $LDL^{\tran}$-form of
the LR-ADI solution. This formulation, while being more natural for projection-based solvers, also turned out to be advantageous when LR-ADI is employed as linear solver for differential matrix equations; see~\cite{LanMS15}.


\section{Shift computation}\label{Shift computation}
Many of the procedures, available in the literature, for the ADI shift computation need the explicit construction of a basis of $\range(Z_j)$ or a subspace thereof.
For instance, in~\cite{Benner2014-2015} the authors suggest to use, as shifts $p_j$, a subset of the Ritz values of $A$ with respect to $\mathcal{Z}_j=\range(\widetilde Z_j)$, where $\widetilde Z_j\in\mathbb{R}^{n\times h}$ consists of the last $h>0$ columns of $Z_j$ that have been orthogonalized with respect to each other. However,~\eqref{eq:lowrankfactor_sol} shows that Algorithm~\ref{LR_ADI_EKSM} provides us with a matrix $Z_j$ such that $\range(Z_j)\subseteq\mathbf{EK}_{m_j}^\square(A,B)$ so that the Ritz values of $A$ with respect to $\mathbf{EK}_{m_j}^\square(A,B)$ can be employed as shifts. Moreover, in standard LR-ADI implementations, one has to explicitly compute the projection of $A$ onto $\mathcal{Z}_j$ increasing the computational efforts of the overall procedure. 
In our approach, the projection of $A$ onto $\mathbf{EK}_{m_j}^\square(A,B)$
is given for free as this amounts to $T_{m_j}$ and no additional operations are required. 

The observation above can be applied to many schemes for the shift computation. In the following we give some details for the residual-Hamiltonian-based shifts and the residual norm-minimizing shifts presented in~\cite{Kuerschner2019}.

In~\cite[Section 2.1.3]{Kuerschner2019}, at the $j$-th LR-ADI iteration, the Hamiltonian matrix $\mathcal{H}_j=\begin{bmatrix}
  A^{\tran} & O \\
  W_{j}W_{j}^{\tran} & -A\\                                                                                                              \end{bmatrix}$ is considered and its projection onto $\mathcal{Z}_{j}$, namely  $ \mathcal{\widetilde H}_j=\begin{bmatrix}
  {(\widetilde Z_{j}^{\tran}A\widetilde Z_{j})}^{\tran} & O \\
  \widetilde Z_{j}^{\tran}W_{j}W_{j}^{\tran}\widetilde Z_{j} & -\widetilde Z_{j}^{\tran}A\widetilde Z_{j}\\                                                                                                              \end{bmatrix}$, is constructed. In our case, we can easily construct the projection of $\mathcal{H}_j$ onto $\mathbf{EK}_{m_j}^\square(A,B)$ and this is given by 
\begin{equation}\label{eq:Hamiltonian_projected}
  \mathcal{\widetilde H}_j=\begin{bmatrix}
    T_{m_{j}}^{\tran} & O \\
    \Upsilon_j^{\tran}\Upsilon_j & -T_{m_{j}}\\                                                                                                              \end{bmatrix}\in\mathbb{R}^{4m_{j}q\times 4m_{j}q}.
\end{equation}
With~\eqref{eq:Hamiltonian_projected} at hand, we compute its stable eigenpairs $\left(\lambda_k,\begin{bmatrix}s_k\\
    t_k\\
  \end{bmatrix}\right)
$, $\text{Re}(\lambda_k)<0$, $s_k,t_k\in\mathbb{R}^{2m_{j}q}$, and {\color{black} the $(j+1)$-st residual-Hamiltonian-based shift $p_{j+1}$ is selected as} the eigenvalue $\lambda_{\widehat k}$ such that $t_{\widehat k}=\argmax\{\|t_k\|\}$. 

For the computation of residual-norm-minimizing shifts, in~\cite[Section 3]{Kuerschner2019} a rather involved optimization procedure is presented. In particular, the real and imaginary parts of $p_{j+1}=\theta_{j+1}+\imath\xi_{j+1}$ are computed by solving the following minimization problem
\begin{equation}\label{eq:minimization_problem_shifts}
  [\theta_{j+1},\xi_{j+1}]=\argmin_{\theta\in\mathbb{R}_-,\xi\in\mathbb{R}}\|W_j-2\theta({(A+(\theta+\imath\xi)I)}^{-1})W_j\|^2.
\end{equation}
The objective function in~\eqref{eq:minimization_problem_shifts} is expensive to evaluate, making the shift computation often more expensive than a single LR-ADI iteration. To overcome this issue, K\"{u}rschner proposes to employ smaller matrices $\widetilde A$ and 
$\widetilde W_j$ in place of $A$ and $W_j$. Once again, $\widetilde A$ and 
$\widetilde W_j$ are the projection of $A$ and $W_j$ onto a suitable subspace. This subspace is chosen to be $\mathbf{EK}_\ell^\square(A,B)\cup\range(Z_j)$ for a certain, usually small, $\ell>0$. 
In our implementation, $\mathbf{EK}_\ell^\square(A,B)\cup\range(Z_j)\subseteq \mathbf{EK}_{m_j}^\square(A,B)$ if $\ell\leq m_j$. Therefore, we can set $\widetilde A=T_{m_j}$ and $\widetilde W_j=\Upsilon_j$ for the approximation of $[\theta_{j+1},\xi_{j+1}]$~\eqref{eq:minimization_problem_shifts}.

\section{Numerical examples}\label{Numerical examples}

In this section we illustrate the potential of the scheme we propose in this paper. The two variants of the LR-ADI-EKSM method, we have illustrated in Section~\ref{Merging the two iterative procedures}, will be denoted by LR-ADI-EKSM(G) and LR-ADI-EKSM(MR). In particular, in LR-ADI-EKSM(G) we solve the linear systems by imposing a Galerkin condition, i.e., the matrix $Y$ is computed by solving the reduced problem~\eqref{eq:projected_linearsystem_general}. In LR-ADI-EKSM(MR), $Y$ solves the least squares problem~\eqref{eq:projected_minimalresidual_general}. 

We test Algorithm~\ref{LR_ADI_EKSM}
on different instances of~\eqref{eq:lyap} coming from the discretization of certain PDEs, and we study how the computational cost of the main steps of  
Algorithm~\ref{LR_ADI_EKSM} depends on the problem dimension $n$ and rank of the
right hand side $q$. 

The results achieved by Algorithm~\ref{LR_ADI_EKSM} are also compared to the ones obtained by running a standard implementation of the LR-ADI method. In particular, we employed the \MATLAB{} function {\tt mess\_lradi} available in the M-M.E.S.S. package~\cite{SaaKB22-mmess-2.2}. Notice that {\tt mess\_lradi} is intended to be a \emph{black-box} routine so that many checks and inspections are performed before the actual solution process starts. This may increase the overall running time of {\tt mess\_lradi}. Therefore, to have fair comparisons, we also report the results obtained by running a standard implementation of LR-ADI where the overhead cost mentioned above is not present. Such a routine is simply denoted by {\tt lradi} in the tables that follow. 

For a better understanding, in Table~\ref{tab0} we summarize the adopted linear system solver included in the tested routines for each of the numerical experiments that follow. Similarly, in Table~\ref{tab0} we indicate whether a given scheme is equipped with the relaxation strategy coming from~\cite{Kuerschner2018} {\color{black} for the selection of $\varepsilon_{\mathtt{inn}}$}.
  \begin{table}[t]
    \centering
    \begin{tabular}{rrrr}
      & & \textsc{Solver} & \textsc{Relaxation} \\
      \hline
    \multirow{3}{*}{\textbf{Experiment 1}}    
   & LR-ADI-EKSM(G) & backslash & \cmark\\
   &{\tt lradi} & backslash & \xmark\\
   &{\tt mess\_lradi} & backslash & \xmark\\
   \hline
    \multirow{3}{*}{\textbf{Experiment 2}}    
   & LR-ADI-EKSM(MR) & PGMRES & \cmark\\
   &{\tt lradi} & PGMRES & \cmark\\
   &{\tt mess\_lradi} & PGMRES & \xmark\\
   \hline
      \multirow{3}{*}{\textbf{Experiment 3}}    
   & LR-ADI-EKSM(G) & backslash & \cmark\\
   &{\tt mess\_lradi} & backslash & \xmark\\
   & K-PIK & backslash & --- \\
   \hline
      \multirow{3}{*}{\textbf{Experiment 4}}    
   & LR-ADI-EKSM(G) & backslash & \cmark\\
   &{\tt mess\_lradi} & backslash & \xmark\\
    \end{tabular}
    \caption{\textsc{Solver}: solver employed for solving the linear systems
      with $A$ in LR-ADI-EKSM and K-PIK and with $A+p_{j}I$ in {\tt lradi} and
      {\tt mess\_lradi}. In the column \textsc{Relaxation} we indicate whether a
      certain scheme is equipped with the relaxation strategy proposed
      in~\cite{Kuerschner2018}.}%
    \label{tab0}
  \end{table}

For all experiments, the tolerance $\varepsilon_{\mathtt{out}}$ for the relative residual norm is set to $10^{-8}$. Moreover, except for Experiment~\ref{Ex.3}, we always employ the residual-Hamiltonian-based shifts presented in~\cite{Kuerschner2019} and computed as illustrated in Section~\ref{Shift computation}. 

All results were obtained by running \MATLAB{} R2020b~\cite{MATLAB} on a 
standard node\footnote{CPU\@: 2x Intel Xeon Skylake Silver 4110 @ 2.1 GHz, 8 cores per CPU\@. RAM\@: 192 GB DDR4 ECC.} of the Linux cluster \texttt{mechthild} hosted at the Max Planck Institute for Dynamics of Complex Technical Systems in Magdeburg, Germany.

\begin{example}\label{Ex.1}
  In the first experiment we consider a Lyapunov equation where 
  \[A=I_h\otimes D_h+D_h\otimes I_h,\quad D_h=\text{tridiag}(1,-2,1)\in\mathbb{R}^{h\times h}.\]
  Therefore, $A\in\mathbb{R}^{n\times n}$, $n=h^2$, is symmetric and stable. 
  We first consider a matrix $B\in\mathbb{R}^{n\times q}$ with random entries and unit norm, and 
  in Table~\ref{tab1.ex1} we depict how the overall solution time distributes among the main steps of our algorithm for different values of $n$ and $q$. 
  
  In both LR-ADI-EKSM(G) and LR-ADI-EKSM(MR), the linear systems with $A$ required for the basis construction are solved by means of the \MATLAB{} sparse direct solver ``backslash''. In particular, $A$ is factorized once and for all before the iterative procedures start so that only triangular systems are actually solved during the basis construction. The computational time for the factorization of $A$ is always included in the results that follow.

  In this experiment, LR-ADI-EKSM(G) and LR-ADI-EKSM(MR) perform very similarly. We thus report only the results achieved by the former. 
  
  \begin{table}[h]
    \centering
    \begin{tabular}{rrrrrrrr}
      & & & \textsc{Basis} & \textsc{Projected Pr.} & \textsc{Shift} & \textsc{Etc} & \textsc{Total} \\
      $n$ & $q$ & {\sc It.} & Time (s) & Time (s) & Time (s) & Time (s) & Time (s) \\
      \hline
      360\,000& 1 & 28 & {\color{black}8.00} & {\color{black}0.07} & {\color{black}1.84} & {\color{black}2.33} & {\color{black}12.24}\\
      & 3 & 27 & {\color{black}21.70} & {\color{black}0.26} & {\color{black}10.35} & {\color{black}3.66} & {\color{black}35.97}\\
      & 5 & 27 & {\color{black}32.79} & {\color{black}0.50} & {\color{black}22.38} & {\color{black}4.60} & {\color{black}60.27}\\
      \hline
      640\,000 & 1 & 29 & {\color{black}16.87} & {\color{black}0.09} & {\color{black}2.59} & {\color{black}4.74} & {\color{black}24.29}\\
      & 3 & 31 & {\color{black}42.74} & {\color{black}0.31} & {\color{black}14.91} & {\color{black}7.45} & {\color{black}65.41}\\
      & 5 & 31 & {\color{black}101.15} & {\color{black}0.81} & {\color{black}34.70} & {\color{black}9.26} & {\color{black}145.92}\\
      \hline
      1\,000\,000 & 1 & 26 & {\color{black}28.99} & {\color{black}0.09} & {\color{black}2.64} & {\color{black}7.52} & {\color{black}39.24}\\
      & 3 & 31 & {\color{black}112.60} & {\color{black}0.45} & {\color{black}19.04} & {\color{black}11.67} & {\color{black}143.76}\\
      & 5 & 29 & {\color{black}183.80} & {\color{black}1.12} & {\color{black}40.56} & {\color{black}14.52} & {\color{black}240.00}\\
      
    \end{tabular}
    \caption{Experiment~\ref{Ex.1}. Computational timings devoted to the
      different main steps of LR-ADI-EKSM(G) for different values of the problem
      size $n$ and rank of the right-hand side $q$. 
      \textsc{Basis}: Basis construction (Algorithm~\ref{LR_ADI_EKSM} --- lines~\ref{basis_construction_1},~\ref{Alg.line:basis2}, and~\ref{basis_construction_3}). \textsc{Projected Pr.}: Computation of $Y$ (Algorithm~\ref{LR_ADI_EKSM} --- lines~\ref{projected_problem_1}, and~\ref{projected_problem_1bis}). 
      \textsc{Shift}: Shifts computation
      (Algorithm~\ref{LR_ADI_EKSM} --- line~\ref{shift_computation_1}). \textsc{Etc}: Remainder of the algorithm {\color{black}(e.g., factorization of $A$)}.
      {\sc It.} indicates the number of ADI iterations that have been implicitly performed.}\label{tab1.ex1}
  \end{table}
  
  As expected, the time devoted to the basis construction represents the
  majority of the overall computational efforts. This is the usual case in
  Krylov projection algorithms. This cost increases as $q$ grows. Indeed, a larger subspace is computed making the basis construction, and in particular the orthogonalization step, rather demanding. Having a large dimensional approximation space leads to a more expensive shift computation, as well.
  
  In Figure~\ref{fig1.ex1} {\color{mycolor1}(left $y$-axis)} we illustrate how the dimension of the computed extended Krylov subspace grows in terms of $j$ for $n=360\,000$ and different values of $q$. 
  
  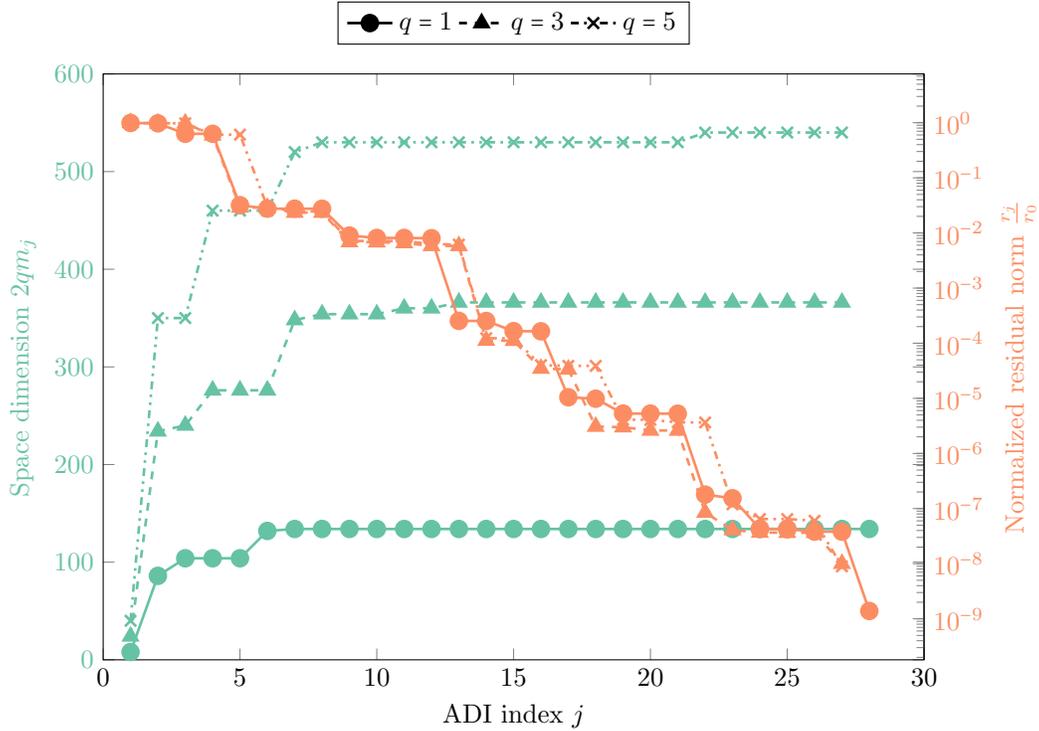
\begin{figure}[h]
    \centering
    \pgfplotscreateplotcyclelist{mylist2}{
{mark options={solid}, mark=*, solid},
{mark options={solid}, mark=triangle*, dashed},
{mark options={solid}, mark=x, dashdotdotted}
}
\begin{tikzpicture}
  \pgfplotstableread{data_figure1.txt}\tablefigureone 
  \pgfplotsset{
    every axis/.append style={
      cycle list name=mylist2,
      xmin = 0, 
      xmax = 30, 
      width = .7\textwidth, 
      height = .5\textwidth, 
      scale only axis,
    },
    legend image post style={black}
  }
  
  \begin{axis}[ 
    ymin = 0, 
    ymax = 600, 
    xtick = {0, 5, 10, 15, 20, 25, 30}, 
    ytick = {0, 100, 200, 300, 400, 500, 600}, 
    xlabel = {ADI index $j$}, 
    ylabel = {\textcolor{mycolor1}{Space dimension $2qm_j$}},
    legend style = {at={(0.5, 1.05)}, anchor=south},
    legend columns = 3
    ]
    \pgfplotsset{
      every axis plot/.append style={mycolor1},
      every y tick label/.append style={mycolor1}
    }
 	\addplot table[x index = 0, y index = 1] {\tablefigureone};
	\addplot table[x index = 3, y index = 4] {\tablefigureone};
 	\addplot table[x index = 6, y index = 7] {\tablefigureone}; 
    \legend{\(q=1\), \(q=3\), \(q=5\)};
  \end{axis}

  \begin{semilogyaxis}[
    axis y line*=right,
    axis x line=none,
    ylabel = {\textcolor{mycolor2}{Normalized residual norm \(\frac{r_{j}}{r_0}\)}}, 
    ]
    
    \pgfplotsset{
      every axis plot/.append style={mycolor2},
      every y tick label/.append style={mycolor2}
    }
    \addplot table[x index = 0, y index = 2] {\tablefigureone};
    \addplot table[x index = 3, y index = 5] {\tablefigureone};
    \addplot table[x index = 6, y index = 8] {\tablefigureone};
  \end{semilogyaxis}
  
\end{tikzpicture} 
    \caption{Experiment~\ref{Ex.1}. {\color{mycolor1}Dimension of the constructed extended Krylov
      subspace} and {\color{mycolor2} computed normalized residual norms} as $j$ grows, i.e.\ the ADI progresses, for problem size $n=360\,000$.}\label{fig1.ex1}
  \end{figure}

  In this experiment, we can notice that the subspace constructed to solve the second shifted linear system, namely $(A+p_2I)S_2=W_1$, is a very rich approximation space in terms of spectral information. Indeed, we need to only slightly expand it 
  to solve the subsequent linear systems {\color{black} without compromising the decrease in the Lyapunov residual norm; see Figure 1} {\color{mycolor2}(right $y$-axis).}
  This means that the majority of the computational efforts are dedicated to solve the second linear system, and we can capitalize on them for $j>2$ reducing the overall workload of the solution process. We would like to mention that such a phenomenon is partially due to the adaptive selection of the inner tolerance $\varepsilon_{\mathtt{inn}}^{(j)}$ coming from~\cite{Kuerschner2018}.
  
  We now compare LR-ADI-EKSM(G) with the function {\tt mess\_lradi} of the
  M-M.E.S.S. package~\cite{SaaKB22-mmess-2.2}, an abstract function handle based
  implementation  of the LR-ADI, and {\tt lradi}, a plain matrix-based
  implementation of the same algorithm.
  
  To this end, we make $B\in\mathbb{R}^n$ the normalized vector of all ones. For having fair comparisons, we employ the shifts computed by the LR-ADI-EKSM(G) in all the different implementations.
  This leads to a very similar trend in the relative residual norm achieved by the routines even though the shifted linear systems in {\tt mess\_lradi} and {\tt lradi} are solved at very high accuracy\footnote{The \MATLAB{} sparse direct solver ``backslash'' is employed for solving $(A+p_{j}I)S_j=W_{j-1}$ for all $j$.}, whereas the relaxation strategy of~\cite{Kuerschner2018} is implemented in LR-ADI-EKSM(G). In Figure~\ref{fig2.ex1}, we report the relative difference between the relative residual norms computed by LR-ADI-EKSM(G) and {\tt mess\_lradi} throughout all the necessary iterations $j$  
  for different problem dimension $n$ {\color{black} along with the values of $\epsilon^{(j)}_{\mathtt{inn}}$ we employed}. In agreement with the results presented in~\cite{Kuerschner2018}, we can notice that {\color{black} the distance between the computed relative residual norms is always rather moderate and smaller than $\epsilon^{(j)}_{\mathtt{inn}}$\footnote{{\color{black}This is true for all the experiments we ran except for $n=640\,000$, at the very last iteration where $r_{27}^{\text{LR-ADI-EKSM}}\approx1.3\times 10^{-8}$ whereas $\epsilon_{\mathtt{inn}}^{(27)}\approx1.6\times 10^{-8}$
}}}. Very similar results are obtained by comparing the residual norms attained by {\tt lradi} in place of {\tt mess\_lradi}.
  
  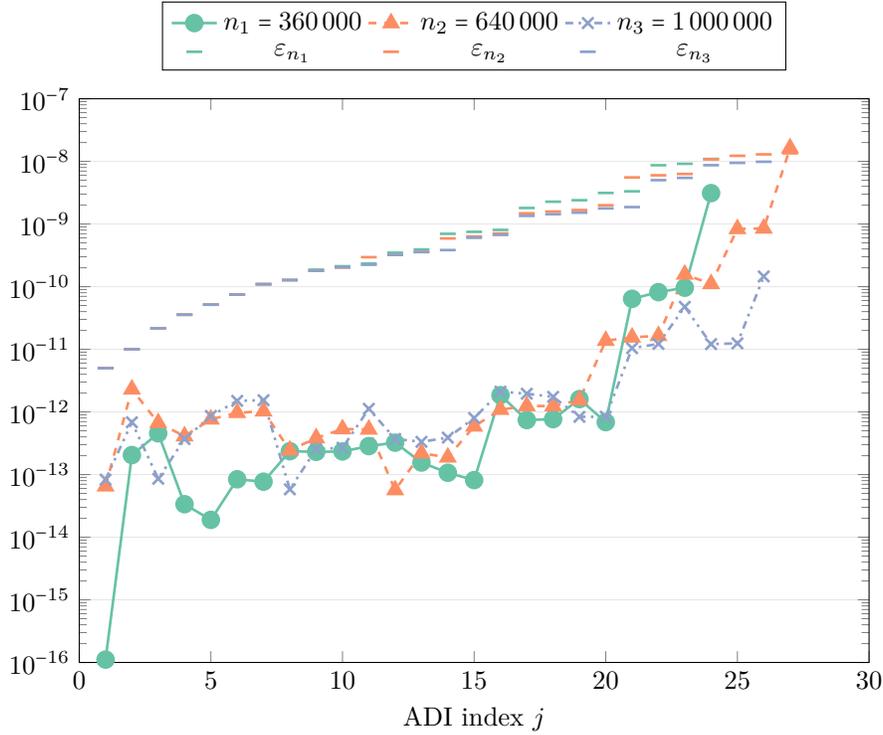
\begin{figure}[h]
    \centering
  \tikzexternalenable%
  \tikzsetnextfilename{figure2}%
  \filemodCmp{figure2.tikz}{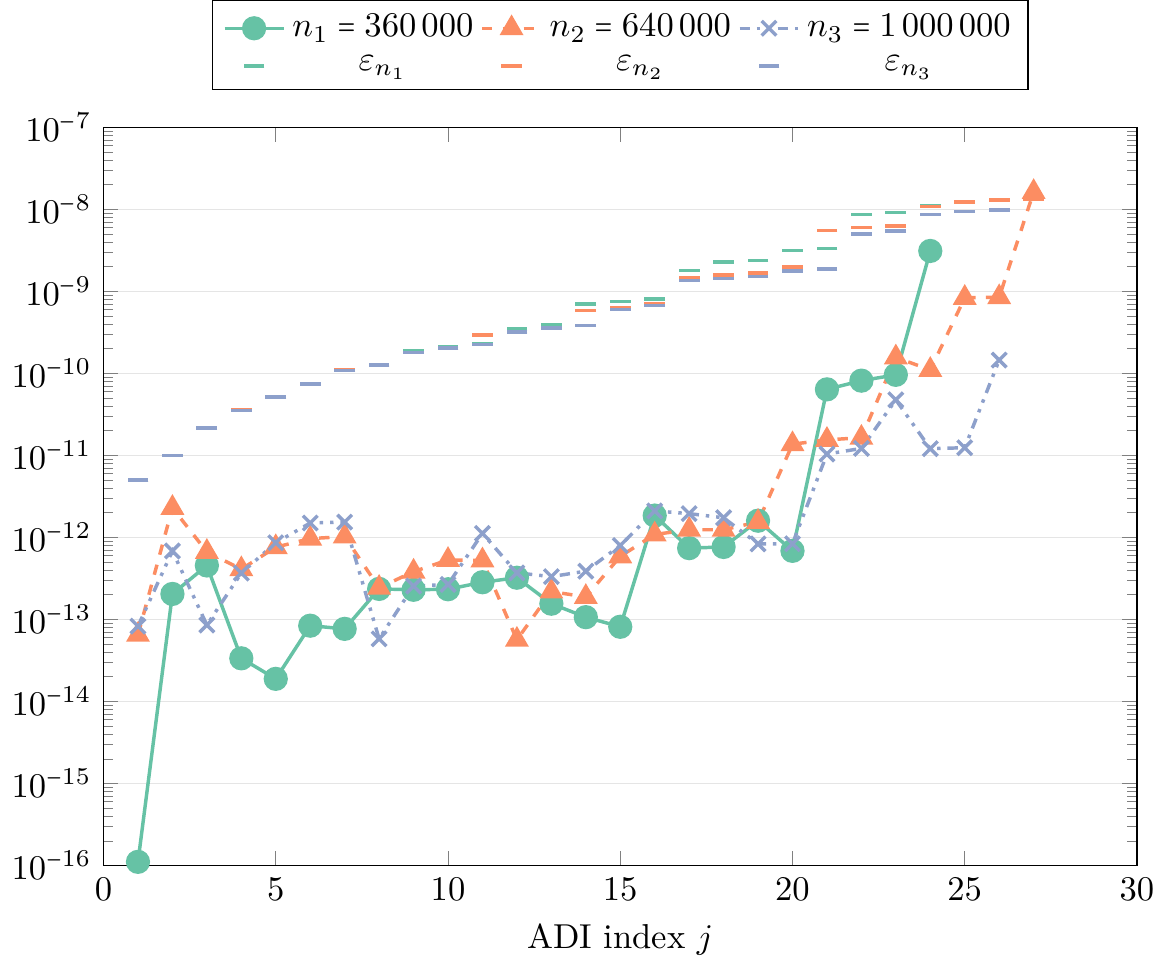}%
  {\tikzset{external/remake next}}{}%
  \begin{tikzpicture}

  \pgfplotstableread{data_figure2.txt}\tablefiguretwo 

  \begin{semilogyaxis}[ 
    width = .7\textwidth, 
    height = .5\textwidth, 
    scale only axis, 
    xmin = 0, 
    xmax = 30, 
    ymin = 1e-16, 
    ymax = 1e-7, 
    xtick = {0, 5, 10, 15, 20, 25, 30}, 
    xlabel = {ADI index $j$}, 
    ymajorgrids,
    legend style = {at={(0.5, 1.05)}, anchor=south},
    legend columns = 3
    ]
 	\addplot table[x index = 0, y index = 1] {\tablefiguretwo};
	\addplot table[x index = 3, y index = 4] {\tablefiguretwo};
 	\addplot table[x index = 6, y index = 7] {\tablefiguretwo}; 
 	\addplot table[x index = 0, y index = 2] {\tablefiguretwo};
	\addplot table[x index = 3, y index = 5] {\tablefiguretwo};
 	\addplot table[x index = 6, y index = 8] {\tablefiguretwo};
	
    \legend{%
		\(n_1=360\,000\), \(n_2=640\,000\), \(n_3=1\,000\,000\),
		\(\varepsilon_{n_{1}}\),\(\varepsilon_{n_{2}}\),\(\varepsilon_{n_{3}}\)
	};
  \end{semilogyaxis} 
\end{tikzpicture} %
  \tikzexternaldisable%

    \caption{Experiment~\ref{Ex.1}. Relative gap
      $\left(\frac{|r_j^{\text{LR-ADI-EKSM}}-r_j^{\texttt{mess\_lradi}}|}%
        {r_j^{\texttt{mess\_lradi}}}\right)$
      between the residual norms $r_j^{\text{LR-ADI-EKSM}}$
      and $r_j^{\texttt{mess\_lradi}}$ computed by LR-ADI-EKSM(G) and {\tt
        mess\_lradi}, respectively, as $j$ grows, i.e.\ ADI converges, and
      different problem sizes $n$, {\color{black}together with the corresponding inner inexact
      solver tolerance \(\varepsilon^{(j)}_{\mathtt{inn}}\), denoted \(\varepsilon_{n}\)
      to relate the problem sizes.}}\label{fig2.ex1}
  \end{figure}
  
  We also compare the routines in terms of computation time. The results are collected in Table~\ref{tab2.ex1}. Since we employ the shifts computed within LR-ADI-EKSM(G) also for {\tt mess\_lradi} and {\tt lradi}, we do not consider the time devoted to the shift computation when reporting the performances of LR-ADI-EKSM(G) in Table~\ref{tab2.ex1}. 
  
  \begin{table}[h]
    \centering
    \begin{tabular}{rrrrr}
      & & LR-ADI-EKSM(G) &{\tt lradi} & {\tt mess\_lradi} \\
      $n$ & {\sc It.} & Time (s) & Time (s) & Time (s) \\
      \hline
      360\,000& 24 & 10.08 & 31.84 & 30.94
      \\
      640\,000& 27 & 19.99 & 65.84 & 66.22 \\
      1\,000\,000& 36 & 39.18 &101.18 & 101.61\\
    \end{tabular}
    \caption{Experiment~\ref{Ex.1}. Computational timings achieved by
      LR-ADI-EKSM(G), {\tt lradi}, and {\tt mess\_lradi} for different problem
      sizes $n$. {\sc It.} indicates the number of ADI iterations that have been
      (implicitly) performed.}%
    \label{tab2.ex1}
  \end{table}
  
  The results in Table~\ref{tab2.ex1} show that, for this experiment, our proposed scheme combined with the relaxation strategy presented in~\cite{Kuerschner2018} leads to a remarkable speed-up of the solution process --- up to 50\% --- when compared to a standard implementation of the LR-ADI method.

\end{example}

\begin{example}\label{Ex.2}\rm
  In the second experiment, we consider a problem similar to~\cite[Example 6]{Palitta2016}. In particular, the matrix $A$ comes from the centered finite difference discretization of the 3-dimensional convection-diffusion operator
  \[
    \mathcal{L}(u)=-\zeta\Delta u+\mathbf{w}\cdot\nabla u,
  \]
  on the unit cube with zero Dirichlet boundary conditions. The convection vector $\mathbf{w}$ is given by $\mathbf{w}=(\phi_1(x)\psi_1(y)\pi_1(z),\,0,\,\pi_3(z))=((1-x^2)yz,\,0,\,e^z)$
  whereas $\zeta>0$. By employing $h$ nodes in each direction, the discretization phase leads to a matrix $A$ that can be written as 
  \[
    A=( D_{h}+\Pi_{3}N^{\tran})\otimes I_h\otimes I_{h}+I_{h}\otimes D_{h}\otimes I_{h}+I_{h}\otimes I_{h}\otimes D_{h}+\Pi_{1}\otimes\Psi_{1}\otimes\Phi_{1}N,
  \]
  where $D_h=\zeta {(h-1)}^{2}\cdot\text{tridiag}(-1,2,-1)\in\mathbb{R}^{h\times h}$, $N=-\frac{(h-1)}{2}\cdot\text{tridiag}(-1,0,1)\in\mathbb{R}^{h\times h}$, and $\Phi_{i}$, $\Psi_{i}$, and $\Pi_{i}$ are diagonal matrices whose diagonal entries correspond to the nodal values of the corresponding functions $\phi_i$, $\psi_i$, and $\pi_i$. See~\cite{Palitta2016} for further details. $B\in\mathbb{R}^n$, $n=h^3$, is a vector with random entries. 

  Due to the 3D nature of the problem, the nonsymmetric linear systems with $A$ involved in the basis construction in LR-ADI-EKSM are solved by GMRES~\cite{Schultz1986}. In particular, we employ the GMRES implementation written by Lund et al~\cite{Kreetal21}, namely the function {\tt bgmres} in~\cite{Lun20}. GMRES is stopped whenever the computed relative residual norm gets smaller than $10^{-10}$.

  It is well-known that (polynomial) Krylov methods for linear systems need to be preconditioned to achieve a fast convergence in terms of number of iterations. To this end, as suggested in~\cite{Palitta2016}, we employ the following preconditioning operator when solving the linear systems with $A$,
  \[
    \mathcal{P}=(D_h+\Pi_3N^{\tran})\otimes I_h\otimes I_h+I_h\otimes D_h\otimes I_{h}+I_{h}\otimes I_{h}\otimes D_{h}+\widebar\pi_{1}I_{h}\otimes\Psi_{1}\otimes\Phi_{1}N,
  \]
  where $\widebar\pi_1$ is the mean value of the function $\pi_1$ in $[0,1]$. At each GMRES iteration, we thus have to invert $\mathcal{P}$, namely we have to compute $\widebar v=\mathcal{P}^{-1}v$ for $v\in\mathbb{R}^n$. This operation is performed by solving the Sylvester equation
  \[
    (D_{h}\otimes I_{h}+I_{h}\otimes D_{h}+\widebar\pi_{1}\Psi_{1}\otimes\Phi_{1}N) \mathbf{\widebar V}+\mathbf{\widebar V}{(D_h+\Pi_3N^{\tran})}^{\tran}=\mathbf{V},
  \]
  where $\mathbf{\widebar V},\mathbf{V}\in\mathbb{R}^{h^2\times h}$ are such that $\text{vec}(\mathbf{\widebar V})=\widebar v$ and $\text{vec}(\mathbf{V})=v$. Since the coefficient matrices in the equation above have moderate dimensions, the Bartels-Stewart method~\cite{Bartels1972} is employed for its solution and the Schur decompositions of the coefficient matrices are computed once and for all before the iterative procedure starts. We always employ a right preconditioning scheme in order to easily have access to the actual residual norm.
  
  Also, for the shifted linear systems with $A+p_{j}I$, within {\ttfamily mess\_lradi} and {\ttfamily lradi}, we employ preconditioned GMRES equipped with the preconditioning operator $\mathcal{P}+p_{j}I$. Once again, this preconditioner is applied by solving the Sylvester equation
  \[
    (D_h\otimes I_h+I_h\otimes D_{h}+\widebar\pi_1\Psi_1\otimes\Phi_1N) \mathbf{\widebar V}+\mathbf{\widebar V}{(D_h+\Pi_{3}N^{\tran}+p_{j}I_{h})}^{\tran}=\mathbf{V}.
  \]
  Even though this is in general a better preconditioner for $A+p_{j}I$ compared to $\mathcal{P}$, its application involves complex arithmetic whenever $\text{Im}(p_{j})\neq 0$ with a consequent increment in the computational efforts devoted to the preconditioning step.
  
  For this experiment, {\tt lradi} is equipped with the relaxation strategy presented in~\cite{Kuerschner2018}.
  
  Also for this experiment, LR-ADI-EKSM(G) and LR-ADI-EKSM(MR) perform very similarly, with LR-ADI-EKSM(MR) achieving slightly better results in terms of computational time. We thus report only the performance of LR-ADI-EKSM(MR).
  
  The results are collected in Table~\ref{tab1.ex2} for different values of $n$ and $\zeta$. In Table~\ref{tab1.ex2} we also report the number of shifts with nonzero imaginary part.
  
  We would like to mention that we ran some experiments with {\tt mess\_lradi}
  where the shifted linear systems were solved by means of the \MATLAB{} sparse
  direct solver ``backslash'' in place of preconditioned GMRES\@. However, for
  this example the potentially higher accuracy of the direct solves did not
  benefit the computation  and the execution times we achieved with
  ``backslash'' could not keep up with the ones reported for GMRES in Table~\ref{tab1.ex2}. We, thus, decided to omit them here.
  
  \begin{table}[t]
    \centering
    \begin{tabular}{rrrcrrrr}
      & & & & LR-ADI-EKSM(MR) &{\tt lradi} & {\tt mess\_lradi} \\
      $\zeta$ & $n$ & {\sc It.} & \#$\{p_j\notin\mathbb{R}\}$& Time (s) & Time (s) & Time (s)\\
      \hline
      0.05 &125\,000 & 20 & 12 & 80.61 & 153.28 & 187.52 \\
      & 512\,000 & 20 & 12 & 812.74 & 1\,342.27 & 2\,161.99 \\
      & 1\,000\,000 & 22 & 10 & 3\,183.54 & 5\,764.21 & 6\,133.54\\
      \hline
      0.005 & 125\,000 & 45 & 44 & 65.23 & 361.58 & 382.11 \\ 
      & 512\,000 & 61 & 60 & 419.45 & 2\,156.94 & 3\,497.51 \\
      & 1\,000\,000 & 67 & 62 & 1\,194.37 & 5\,802.41 & 10\,517.13 \\
      
    \end{tabular}
    \caption{Experiment~\ref{Ex.2}. Computational timings achieved by
      LR-ADI-EKSM(MR), {\tt lradi}, and {\tt mess\_lradi} for different problem
      sizes $n$ and diffusivities $\zeta$. {\sc It.} indicates the number of ADI iterations that have been (implicitly) performed.}\label{tab1.ex2}
  \end{table}
  
  From the results in Table~\ref{tab1.ex2} we can see that LR-ADI-EKSM(MR) is
  very competitive and always achieves computational timings that are
  significantly smaller than the ones required by  {\tt mess\_lradi}. Thanks to the relaxation procedure coming from~\cite{Kuerschner2018}, {\tt lradi} performs better than {\tt mess\_lradi}. 
  
  The performance of all the tested routines is strictly related to the number of complex shifts needed to converge. 
  When this is sizable with respect to the total number of iterations, many of the $n\times n$ linear systems $A+p_{j}I$ within {\tt mess\_lradi} and {\tt lradi} involve complex arithmetic, whereas this is needed only in the solution of the small dimensional least squares problem for the computation of $Y$ in LR-ADI-EKSM(MR).

  We notice that, for a fixed $n$, the computational time of LR-ADI-EKSM(MR) decreases, in general, by reducing $\zeta$, even tough the number of LR-ADI iterations that are implicitly performed increases.
  This is due to the computational efforts required by the solution of the linear systems with $A$ during the basis construction. Indeed, for $\zeta=0.05$, many more GMRES iterations are required than what is necessary for $\zeta=0.005$. In Figure~\ref{fig1.ex2}, we report the number of GMRES iterations needed to solve the linear system with $A$ at each $m$, namely every time a new basis vector of the adopted extended Krylov subspace needs to be computed.

  \begin{figure}[tbp]
    \centering
  \tikzexternalenable%
  \tikzsetnextfilename{figure3}%
  \filemodCmp{figure3.tikz}{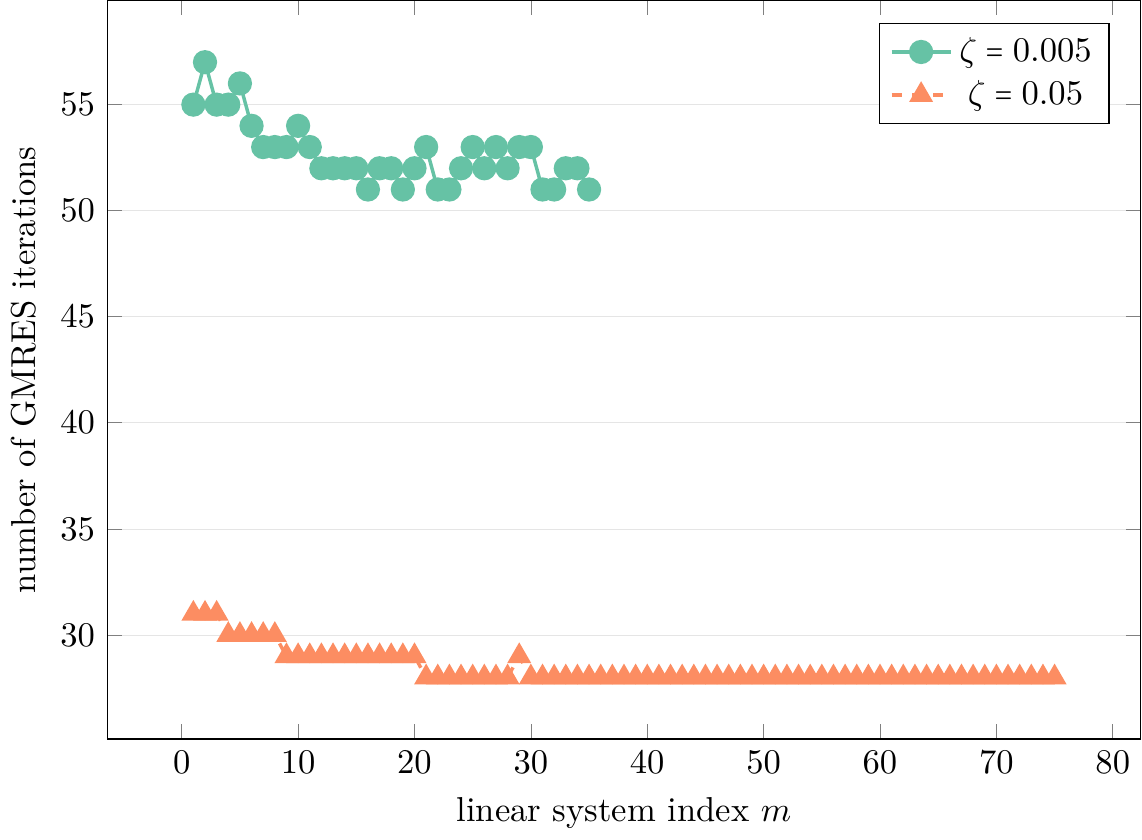}%
  {\tikzset{external/remake next}}{}%
  \begin{tikzpicture}

  \pgfplotstableread{data_figure3.txt}\tablefigurethree 

  \begin{axis}[ 
    width = .7\textwidth, 
    height = .5\textwidth, 
    scale only axis, 
    xlabel = {linear system index $m$}, 
    ylabel = {number of GMRES iterations},
	ymajorgrids,
    legend pos = north east,
    ]
 	\addplot table[x index = 0, y index = 1] {\tablefigurethree};
	\addplot table[x index = 2, y index = 3] {\tablefigurethree};
    \legend{\(\zeta=0.005\), \(\zeta=0.05\)};
  \end{axis} 
\end{tikzpicture} %
  \tikzexternaldisable%

    \caption{Experiment~\ref{Ex.2}. Number of GMRES iterations needed to solve
      the linear systems with $A$ during the basis construction in
      LR-ADI-EKSM(MR) for different values of the diffusivity $\zeta$ 
      and $n=125\,000$.}\label{fig1.ex2}
  \end{figure}
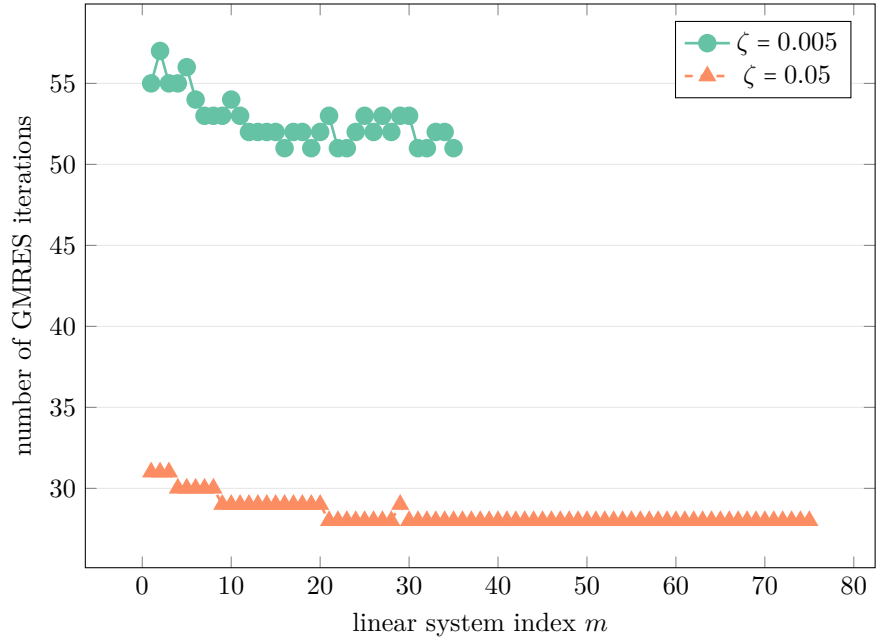

  A rather large number of GMRES iterations is required for solving the linear systems with $A$ in case of $\zeta=0.05$ making the construction of the basis of $\mathbf{EK}^\square_m(A,B)$ more demanding. On the other hand, few GMRES iterations are sufficient to meet the prescribed accuracy for $\zeta=0.005$ and the overall solution procedure turns out to be very successful.
\end{example}

\begin{example}\label{Ex.3}\rm
  In this experiment we compare LR-ADI-EKSM also with
  K-PIK~\cite{Simoncini2007}, since the two routines construct the same
  subspace\footnote{The implementation of K-PIK we employed will be available in the next M-M.E.S.S. release, along with other projection methods for matrix equations. Such implementation is equivalent to the one that can be found on Simoncini's webpage, \url{http://www.dm.unibo.it/~simoncin/software.html}.}. We consider the thermal part of the thermo-elastic modeling of a
  building-block of an experimental machine tool given by the following heat
  equation
  \begin{equation}\label{Ex.3.eq:1}
    \arraycolsep=1.4pt\def\arraystretch{2}
    \left\{
      \begin{array}{rll}
        c_p\rho\frac{\partial T}{\partial t}&=&\lambda\Delta T,\quad \text{in }\Omega,\\
        \lambda\frac{\partial T}{\partial \textbf{n}}&=&f,\quad \text{on }\Gamma_c\subset\partial\Omega, \\
        \lambda\frac{\partial T}{\partial \textbf{n}}&=&\alpha(T_{ext}-T),\quad \text{on }\Gamma_{ext}\subset\partial\Omega, \\
        T(0)&=&0.\\
      \end{array}\right.
  \end{equation}
  \begin{figure}[t]
    \centering
    \includegraphics[width=.5\linewidth]{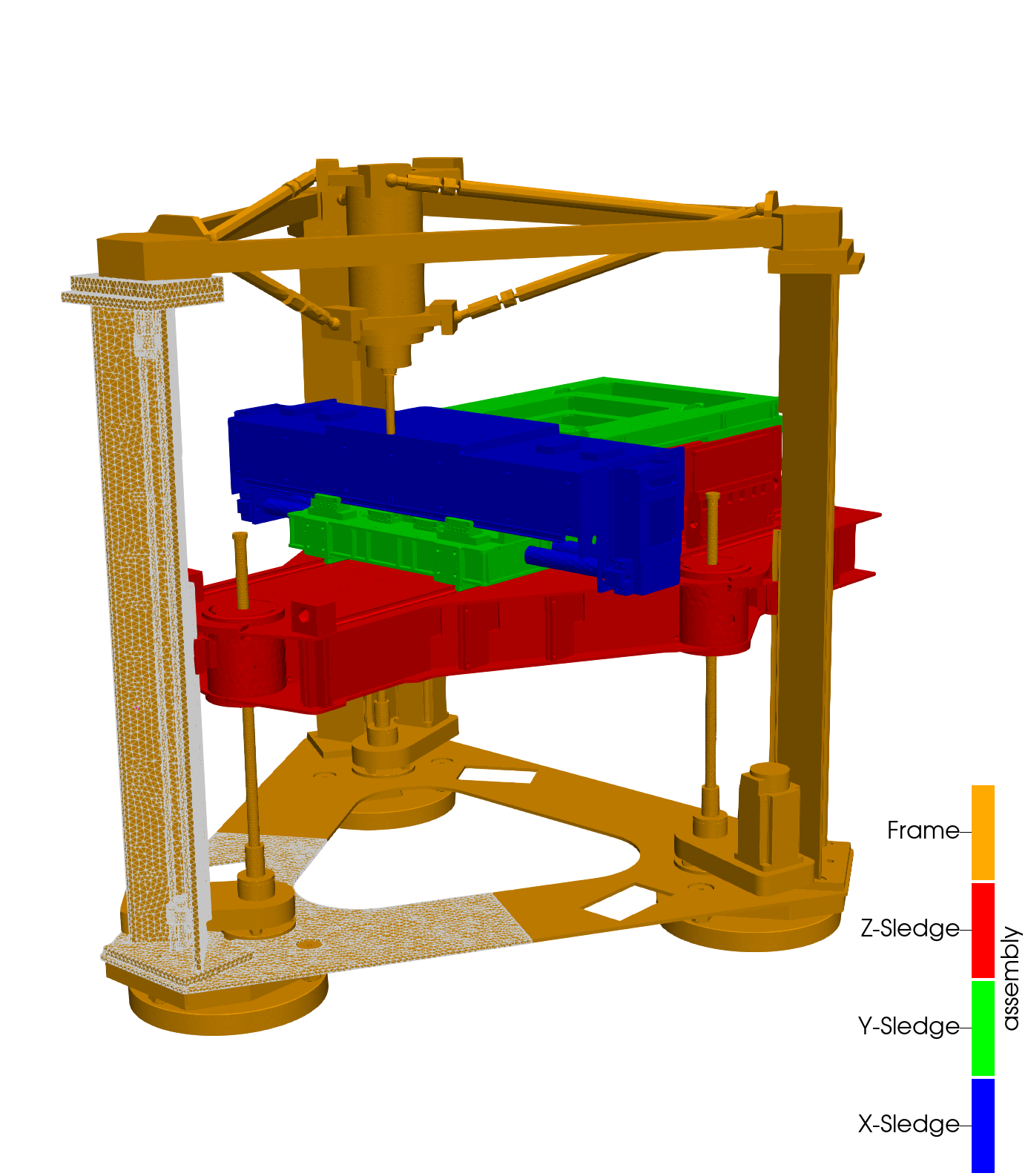}
    \caption{Experiment~\ref{Ex.3}. Finite element grid of the machine frame
        indicated on the CAD model of the full machine. (Source: DFG
        CRC/TR-96 (\url{https://transregio96.de}))}\label{fig:modelCRC}
  \end{figure}
  The discretization in space using the finite element method (here applying the
  proprietary tool ANSYS\footnote{\url{https://www.ansys.com/}}) on the three-dimensional domain, given by the machine frame indicated in
  Figure~\ref{fig:modelCRC}, leads to the LTI system 
  \begin{equation}\label{eq:CRCLTI}
    E\dot T=\left(A-\sum_{i=1}^t \alpha_{i}F_{i}\right) T+Bu(t).
  \end{equation}
  Here, \(A\) represents the discretized Laplacian together with the Robin
  boundary contributions from \(\Gamma_{ext}\) and represented by \(F_i\), while \(B\) results from the
  external control inputs (heats fluxes, e.g.\ induced by the drive motors) on
  \(\Gamma_{c}\). Note that the elastic part of the thermo-elastic model can be
  encoded entirely in the output equation of the corresponding dynamical system
  and is, thus, not relevant here~\cite{morLanSB14}.
  The algebraic problem resulting from this system amounts to a Lyapunov
  equation of the form~\eqref{eq:lyap_gen}. However, due to mass lumping in ANSYS,
  the mass matrix $E$ is diagonal and SPD\@. We can, thus, easily invert its
  square root and consider the Lyapunov equation  
  \[
    E^{-\frac{1}{2}}\left(A-\sum_{i=1}^t \alpha_{i}F_{i}\right)E^{-\frac{1}{2}}\widetilde X +
    \widetilde XE^{-\frac{1}{2}}{\left(A-\sum_{i=1}^t \alpha_{i}F_{i}\right)}^{\tran}E^{-\frac{1}{2}}+E^{-\frac{1}{2}}BB^{\tran}E^{-\frac{1}{2}}=0,\quad \widetilde X=(E^{\frac{1}{2}}XE^{\frac{1}{2}}).
  \]
  So, again, we can efficiently retract to a problem of the
  form~\eqref{eq:lyap}.  Once a low-rank approximation
  $\widetilde Z\widetilde Z^{\tran}$ to $\widetilde X$ is computed, the low-rank
  factor $Z$ such that $ZZ^{\tran}\approx X$ can be retrieved by performing
  $Z=E^{-\frac{1}{2}}\widetilde Z$.

  The actual machine frame in Figure~\ref{fig:modelCRC} consists of several
  parts itself, which are discretized separately. This leads to differently
  sized models of the structure in~\eqref{eq:CRCLTI}. These are reflected
  by the rows of Table~\ref{tab1.ex3}. Accordingly, we solve the Lyapunov
  equation considering different configurations of the PDE~\eqref{Ex.3.eq:1},
  respectively the LTI system in~\eqref{eq:CRCLTI}. In particular, this allows
  us to vary the
  number of degrees of freedom employed in the discretization phase, leading
  to different problem dimensions $n$, modify the Neumann boundary conditions
  obtaining diverse matrices $F_i$, and consider different values for the
  rank $q$ of $B$. Moreover, we set $\alpha_i=10$ for all $i=1,\ldots,t$. 

  The results are collected in Table~\ref{tab1.ex3}. It turns out that the
  Wachspress ADI shifts~\cite{Wachspress2013,Li2002} are particularly effective
  for this experiment, since \(A\) as well as all the \(F_{i}\) and thus
  \(E^{-\frac{1}{2}}\left(A-\sum_{i=1}^t\alpha_{i}F_{i}\right)E^{-\frac{1}{2}}\)
  are symmetric, i.e.\ the spectrum is real. These are the ideal circumstances
  for Wachspress shifts. We, thus, employ those shifts in LR-ADI-EKSM(G) and
  {\tt mess\_lradi}.
  \begin{table}[t]
    \centering
    \begin{tabular}{rrrrrrrrrrr}
      & &  & \multicolumn{3}{r}{LR-ADI-EKSM(G)} &\multicolumn{2}{r}{{\tt mess\_lradi}} & \multicolumn{3}{r}{K-PIK}\\
      $n$ & $t$ & $q$ & {\sc It.} & $\text{dim}\left(\mathbf{EK}^\square\right)$ & Time (s) & {\sc It.} &Time (s) & {\sc It.} & $\text{dim}\left(\mathbf{EK}^\square\right)$ & Time (s) \\
      \hline
      4\,813 & 1& 23  & 54 &644  & 2.73 & 54 & 7.93 & 15 & 736& 8.13\\
      13\,551 &2 & 5 & 53 & 430& 8.89 & 53  &18.09  & 43& 440 &24.93 \\
      25\,872 & 1 & 10 & 63 & 1060&34.32& 63  & 63.37  & 53 & 1080 & 97.74 \\ 
    \end{tabular}
    \caption{Experiment~\ref{Ex.3}. Computational timings achieved by
      LR-ADI-EKSM(G), K-PIK, and {\tt mess\_lradi} for different values of
      problem size $n$, number of Robin boundary conditions~$t$, and rank of the right-hand side $q$. {\sc It.} indicates the number of ADI/K-PIK iterations that have been (implicitly) performed.}\label{tab1.ex3}
  \end{table}
  
  For this experiment, the LR-ADI method, either based on our new formulation or
  on a standard scheme as the one in {\tt mess\_lradi}, turns out to be more
  efficient in terms of computational time than K-PIK\@. Indeed, in spite of the
  smaller number of iterations needed to converge, the large dimension of the
  extended Krylov subspace constructed by K-PIK leads to a rather costly
  solution of the projected equations. Also LR-ADI-EKSM(G) requires the
  construction of an extended Krylov subspace whose dimension is similar to the
  one computed by K-PIK\@. However, if
  $\text{dim}\left(\mathbf{EK}^\square_m(E^{-\frac{1}{2}}\left(A-\sum_{i=1}^t
      \alpha_{i}F_{i}\right)E^{-\frac{1}{2}},E^{-\frac{1}{2}}B)\right)=2mq$, the
  computational cost of solving the inner problems within LR-ADI-EKSM(G) is
  $\mathcal{O}(4m^2q^2)$ floating-point operations (FLOPs) whereas it amounts to $\mathcal{O}(8m^3q^3)$ FLOPs for K-PIK.%

  We conclude by mentioning that in this experiment we relied on the ease of computing $E^{-\frac{1}{2}}$. However, it may happen that the mass matrix $E$ cannot be easily manipulated, e.g., it can be possibly singular, so that the routine presented in this paper cannot be readily applied as we have done in this experiment. 
  We plan to extend the LR-ADI-EKSM framework to this more challenging class of equations in the near future.

\end{example}

\begin{example}\label{Ex.4}\rm
In the last experiment, we show that the proposed framework still needs some
further improvements to efficiently deal with generalized Lyapunov equations of
the form~\eqref{eq:lyap_gen} where the mass matrix $E$ is not diagonal.
To this end, we consider the {\em Steel Profile\/} data set~\cite{morwiki_steel,morBenS05} from the {\sc MORwiki} repository~\cite{morWiki}. 

We compute the observability Gramian of the system, namely the solution $X$ to the equation
\begin{equation}\label{eq:EX4_gen}
A^{\tran}XE+E^{\tran}XA+C^{\tran}C=0,
\end{equation}
where $A\in\mathbb{R}^{n\times n}$ is symmetric negative definite, $C\in\mathbb{R}^{q\times n}$, $q=6$, and $E\in\mathbb{R}^{n\times n}$ is SPD but not diagonal. See~\cite{BenS05b} fur further details on the model.

If $E=LL^{\tran}$ denotes the Cholesky factorization of $E$, we consider the transformed equation 
\begin{equation}\label{eq:EX4_trans}
(L^{-1}A^{\tran}L^{-\tran})\widetilde X+\widetilde X(L^{-1}AL^{-\tran})+L^{-1}C^{\tran}CL^{-\tran}=0,\quad \widetilde X=L^{\tran}XL, 
\end{equation}
and, due to symmetry of \(A\), employ the extended Krylov subspace $\mathbf{EK}_m^\square(L^{-1}AL^{-\tran},L^{-1}C^{\tran})$ as approximation space. Notice that the matrix $L^{-1}AL^{-\tran}$ does not need to be explicitly constructed. See, e.g.,~\cite[Example~5.4]{Simoncini2007}. As before, once $\widetilde Z\widetilde Z^{\tran}\approx\widetilde X$ is computed, we obtain a low-rank approximation to the original $X$ by performing $Z=L^{-\tran}\widetilde Z$.

In Table~\ref{tab1.ex4} we report the results achieved by LR-ADI-EKSM(G) and {\tt mess\_lradi} for different values of $n$.

\begin{table}[t]
    \centering
    \begin{tabular}{rrrrrrrr}
      &  \multicolumn{4}{r}{LR-ADI-EKSM(G)} &\multicolumn{3}{r}{{\tt mess\_lradi}} \\
      $n$  & {\sc It.} & $\text{dim}\left(\mathbf{EK}^\square\right)$ & $\text{rank}(X)$ & Time (s) & {\sc It.} &$\text{rank}(X)$ & Time (s)  \\
      \hline
       20\,209 & 30 & 564 &  180& 7.07&30   & 180 & 0.54  \\
      79\,841 & 31  & 816  & 186 & 34.09  & 31   & 186  & 2.99 \\  
    \end{tabular}
    \caption{Experiment~\ref{Ex.4}. Computational timings achieved by
      LR-ADI-EKSM(G) and {\tt mess\_lradi} for different values of the problem
      size $n$. The running time devoted to the shift computation is not included. {\sc It.} indicates the number of ADI iterations that have been (implicitly) performed.}\label{tab1.ex4}
  \end{table}
  
  From the results in Table~\ref{tab1.ex4} we can readily see that the standard scheme of the LR-ADI method implemented in
  {\tt mess\_lradi} is much faster than LR-ADI-EKSM(G). This is due to the fact that the latter algorithm needs to construct a quite large subspace to achieve the prescribed accuracy with a consequent increment in the computational efforts of the overall procedure. 
  
  We also mention that the rank of the approximate solution computed by LR-ADI-EKSM(G)
  is much lower than the
  dimension of the constructed subspace. We believe that the transformation we performed in~\eqref{eq:EX4_trans}, and thus the employment of $\mathbf{EK}_m^\square(L^{-1}AL^{-\tran},L^{-1}C^T)$, may lead to some spectral redundancy in the adopted approximation subspace and a slower convergence of the method. On the other hand, {\tt mess\_lradi} is able to deal with the original formulation~\eqref{eq:EX4_gen} of the problem. 
  
  To address generalized equations of the form~\eqref{eq:lyap_gen}, 
  we plan to study the employment of different techniques within the Krylov LR-ADI framework we presented in this paper. In particular, the use of nonstandard inner products and (extended) generalized Krylov subspaces~\cite{LiY03} will be explored.
  
\end{example}

\section{Conclusions}\label{Conclusions}
A new formulation of the LR-ADI algorithm for large-scale standard Lyapunov
equations has been proposed.  The computational core of the LR-ADI scheme
consists in the solution of a shifted linear system at each iteration. We showed
that the extended Krylov subspace method can be a valid candidate for this
task. In particular, we described how only one extended Krylov subspace needs to
be constructed to solve all the necessary linear systems required by the LR-ADI
method. The LR-ADI iteration has been completely merged into the extended
Krylov subspace method for shifted linear systems resulting in a novel,
efficient solution procedure. We also showed that many state-of-the-art
algorithms for the shift computation can be easily integrated into our new
scheme. Numerical results demonstrate the potential of our novel algorithm,
especially when this is equipped with the relaxation strategy proposed
in~\cite{Kuerschner2018}, and many complex shifts are needed to converge.

In future work we will consider more involved Lyapunov equations of the form~\eqref{eq:lyap_gen} that cannot be easily transformed into~\eqref{eq:lyap}. While standard implementations of the LR-ADI method naturally address such a scenario by solving linear systems of the form $A+p_{j}E$, further care has to be taken to employ the scheme we presented in this paper. Indeed, the shifted Arnoldi relation~\eqref{eq:shiftedArnoldi} can no longer be exploited. The use of non-standard inner products and generalized Krylov subspace methods~\cite{LiY03} will be investigated.

The framework presented in this paper can be generalized to enhance other LR-ADI-like algorithms for matrix equations. For instance, the LR-ADI method for Sylvester equations~\cite{Benner2009}, or LR-RADI schemes for Riccati equations~\cite{Benner2018,Bertram2020} can be equipped with a procedure similar to the one we proposed here.

\section*{Acknowledgments}
The second author is member of the Italian INdAM Research group GNCS.%

The work presented in this paper has been carried out when the second author was affiliated with the Research Group Computational Methods in Systems
  and Control Theory (CSC),
  Max Planck Institute for Dynamics of Complex Technical Systems, Sandtorstra\ss{e} 1, 39106 Magdeburg, Germany.
  
  The \emph{Steel Profile} dataset is available in the {\sc MORwiki} repository~\cite{morWiki}.
  All the other datasets generated during and/or analysed during the current study are available from the corresponding author on reasonable request.

\bibliographystyle{siamplain}
\bibliography{Krylov_ADI}

\end{document}